\documentclass{article}
\usepackage{mathtools}
\usepackage{amsfonts}
\usepackage{amsmath}
\usepackage{amsthm}
\usepackage{amssymb}
\usepackage{mathrsfs}
\usepackage{graphicx}
\usepackage{caption}
\usepackage{dsfont}
\usepackage{esint}
\usepackage{xcolor}
\usepackage{cases}
\usepackage{authblk}

\theoremstyle{plain}
\newtheorem{theorem}{Theorem}[section]

\newtheorem{lemma}[theorem]{Lemma}

\newtheorem{proposition}[theorem]{Proposition}

\theoremstyle{definition}
\newtheorem{definition}[theorem]{Definition}

\newtheorem{remark}[theorem]{Remark}

\title{Nonlinear stability in a free boundary model of active locomotion}

\author[1]{Leonid Berlyand}
\author[2]{C. Alex Safsten}
\author[3]{Lev Truskinovsky}

\affil[1]{Department of Mathematics and Huck Institute for Life Sciences, The Pennsylvania State University, USA}
\affil[2]{Department of Mathematics, University of Maryland, USA,}
\affil[3]{ESPCI Paris, France}

\begin{document}
	
	\maketitle
	
	\abstract{ Contraction-driven self-propulsion of a large class of living cells can be  modeled by a  Keller-Segel system with free boundaries. The ensuing  ``active'' system, exhibiting both dissipation and anti-dissipation, features  stationary and traveling wave solutions. While the former represent  static cells, the latter describe propagating pulses (solitary waves) mimicking the autonomous  locomotion of the same cells.  In this paper we provide the first    proof of the asymptotic nonlinear stability of both of such   solutions, static and dynamic.  In the case of stationary solutions, the linear stability   is established  using the spectral theorem for compact, self-adjoint operators, and thus linear stability is determined classically, solely by eigenvalues. For traveling waves the picture is   more complex because the linearized  problem is non-self-adjoint, opening the possibility of a  ``dark'' area in the phase space which is not ``visible'' in the purely eigenvalue/eigenvector approach. To establish   linear stability in this case we   employ spectral methods together with the Gearhart-Pruss-Greiner (GPG) theorem, which controls the entire spectrum via bounds on the resolvent operator. For both stationary and traveling wave solutions, nonlinear stability is then proved by showing how the nonlinear part of the problem may be dominated by the linear part and then employing a Gr\"onwall inequality argument. The developed novel methodology  can prove useful also in other problems involving  non-self-adjoint (non-Hermitian or non-reciprocal) operators which are ubiquitous in the modeling of ``active'' matter.
%
%
%
}
	
	\section{Introduction}
	
 The ability of cells to self-propel is fundamental for many aspects of development, homeostasis, and disease, for instance,  cells need to move to form tissues and  their migration  is  also critical during tissue repair   \cite{de2024follow,sengupta2021principles,shellard2020all,weissenbruch2024actomyosin}.   
  The active machinery   behind self-propulsion  resides in the cytoskeleton---a meshwork of actin filaments with contractile cross-linkers represented  by  myosin motors.
 The main active processes in the cytoskeleton are  the   polymerization of actin fibers and  the   relative sliding of actin fibers induced by myosin motors \cite{alberts2002molecular}. The molecular and biochemical basis of these processes  is basically known, however the corresponding  mathematical  theory  is still under development and a variety of multiscale simulation approaches targeting various cell motility mechanisms can be found in the literature   \cite{barnhart2015balance,bray2000cell,calvez2012analysis,carlsson2008mathematical,giomi2014spontaneous,kimpton2013multiple,
  mogilner2009mathematics,
  rafelski2004crawling,wang2012computational,ziebert2013effects}.
 
Aiming at the development of a rigorous mathematical approach to  stability analysis of such models, we  focus  in this paper  on the simplest phenomenon of  self-propulsion in a particular class of cells, keratocytes.  They move by advancing the front through   polymerization    with a simultaneous formation of adhesion clusters.   After the adhesion of the protruding part of the cell is secured, the cytoskeleton contracts due to activity of myosin motors. This  contraction leads to detachment at the rear and depolymerization of the actin network. All three components of the motility mechanism (polymerization, contraction, and adhesion) depend upon continuous  ATP hydrolysis and require intricate regulation by complex signaling pathways involving chemical and mechanical feedback loops  \cite{barnhart2011adhesion,tjhung2015minimal}.

 Contractile force generation is of fundamental importance  for this mode of cell migration.   Using actin fibers as a substrate, myosin motors \cite{howard2002mechanics} generate forces  which  are ultimately responsible for  both the motility initiation and the steady locomotion of   keratocytes \cite{agarwal2019diverse,cowan2022non,biophysica2040046,maree2006polarization, vicente2009non}. In view of   such  central role  of   active contraction and to achieve relative analytical transparency of the mathematical analysis, we consider in this paper  a prototypical model  which emphasizes contraction as the main driving mechanism  while accounting for  polymerization and adhesion only in a  schematic manner.


Our minimal model  of  cell  motility is based on a one-dimensional projection of the complex intracellular dynamics onto the direction of motion. Specifically, we  assume that the motor part of a   cell can be viewed as a one-dimensional continuum with  two free boundaries representing the front and the rear of the moving cell.    We make a simplifying physical assumption that actin polymerization and de-polymerization can take place only on these boundaries and that  these phenomena can be modeled as an influx of mass at the front boundary and its disappearance at the rear boundary.  The  adhesion is also treated in an over-simplified form   as passive  spatially inhomogeneous viscous friction. Instead,  the actomyosin   contraction, which is the main player,  is represented by active spatially inhomogeneous prestress \cite{kruse2006contractility,juelicher2007active}.

As it was first shown in \cite{RecPutTru2013,RecPutTru2015}, the mathematical model,   which captures all these physical   effects  while being amenable to rigorous  mathematical analysis,  reduces to  the one dimensional Keller-Segel system with free boundaries.
 In contrast to the conventional chemotaxic  Keller-Segel model \cite{keller1971model},  here the same set of equations emerges in a purely mechanical setting, see \cite{bois2011pattern,callan2013active,dimilla1991mathematical, kruse2006contractility,larripa2006transport,juelicher2007active, rubinstein2009actin} for the earlier insights along the same lines.  In our Section \ref{sec:model}, where we present for convenience a short derivation of this model, we also highlight its universality (minimality) by showing that   it can be obtained starting from rather different physical assumptions.

    It is important to mention that  alternative free-boundary-type models of cell motility, emphasizing various  other components  of the self-propulsion machinery,   have been used in  numerical simulations     \cite{KerEtAl2008,NicEtAl2017,NwoCam2023,RubJacMog2005,SchSteDuk2010} and, in some cases, also   subjected to rigorous mathematical analysis  
	 \cite{CucMelMeu2020,CucMelMeu2022}.  Closely related  free boundary models describing  tumor growth    have been also   studied both analytically and numerically \cite{NoePer2021,FriHu2006,FriHu2007,HaoEtAl2012}.   Our paper  differs from all this  mathematical work on free boundary modeling of locomotion in its emphasis on  the non-self-adjoint property of the linearized operator resulting from both  nonlocality \cite{Ion1979} and activity \cite{ashida2020non}.  Note that   phase field models of cell motility, representing a mathematical  proxy to our free boundary formulation (front capturing instead of front tracking \cite{de2004front,bernacki2024vertex}), have   been also a subject of extensive  research efforts   \cite{BerFuhRyb2018,BerPotRyb2016,BerPotRyb2017,YanJolLev2019,ZieAra2016}. However, while the corresponding   models allow for very efficient numerical simulations, they are usually not as readily amenable for rigorous stability analysis, and therefore will not be addressed in the present purely analytical study. 
	
 The  one dimensional Keller-Segel system with free boundaries is known to  possesses a family of pulse-like traveling wave solutions, which describe  steady autonomous locomotion of individual cells \cite{RecPutTru2013,RecPutTru2015,RybBer2023,SafRybBer2022}.  These solutions, which can be interpreted as solitary waves,  bifurcate from a family of stationary (static) solutions, representing non-moving cells. The role of  bifurcation parameter  is played by  a non-dimensional    measure of  the level of internal activity with both static and dynamic solutions being 'active' in the sense that they consume and dissipate energy. In  \cite{RecPutTru2013,RecPutTru2015} the  whole variety of stationary  solutions was constructed  analytically and the nature of the corresponding static-dynamic bifurcation  was determined  using   weakly nonlinear analysis   involving a standard approach based on Lyapunov-Schmidt reduction \cite{guo2013bifurcation};  significant numerical evidence that traveling waves bifurcating from homogeneous stationary states have  finite reserve of   stability  was  also obtained.  In \cite{RybBer2023,SafRybBer2022}, the same bifurcation between stationary and traveling wave solutions was studied  in two dimensions, and the  configurations  of the traveling wave solutions  were computed    both analytically (close to the bifurcation point)   and numerically (away from it).  Linear stability was addressed for both stationary and traveling wave solutions with  the  eigenvalue-based stability condition computed explicitly. 

	   The present paper  begins with the rigorous analysis of the static-dynamic bifurcation   via the Crandall-Rabinowitz (CR) theorem \cite{crandall1980mathematical}; note that  the strictly analogous approach in more than 1D would fail due to a lack of conformal invariance.   Then we show analytically that all eigenvalues of traveling wave solutions have negative real part which  complements  the earlier numerical result  that the stability-defining eigenvalue  has negative real part \cite{RybBer2023,SafRybBer2022}. Our main result, though,  is the proof of the  asymptotic stability of both stationary states and traveling waves for appropriate parameter values. 
	   
	   The  main difficulty  in the stability  analysis  of the traveling wave solutions resides  in the  non-self-adjoint (non-Hermitian, or non-reciprocal) nature of the corresponding  linearized operator \cite{Dol1961}, which is an important general  feature of PDE models of ``active" matter \cite{ashida2020non,dinelli2023non,duan2023dynamical, shaat2023chiral, wu2023active}. It is known, for instance, that for non-self-adjoint (NSA) operators, eigenvectors do not necessarily span the entire domain of the operator. Therefore, common stability analysis, e.g. \cite{AleBlaCas2019,OhtTarSan2016}, based only on eigenvalues and eigenvectors  may not be sufficient \cite{Tre2020}. 
	   
	   We recall that  when the linearized problem  is self-adjoint,  the eigenmodes of the stable system can be divided into stable (corresponding to eigenvalues with negative real part), and  center  (with zero real part eigenvalues). In the nonlinear setting, solutions in the corresponding stable manifold would then be  controlled (bounded) by solutions in the center manifold. Furthermore, a nonlinear ODE can be derived for solutions in the center manifold, from which it can be  shown that all such solutions asymptotically approach the equilibrium.  It would then mean that all other solutions also approach it.  The  key   assumption in this approach  to stability   is that eigenvectors of the linearized operator span the entire domain of the operator. This  may not be the case for NSA operators which typically exhibit a  ``dark'' area in the phase space which is not ``visible'' in the purely eigenvalue/eigenvector approach   We address this challenge using directly resolvent analysis  instead of  relying solely  on eigenvalues.\footnote{It should be noted  that  in some specific NSA problems  eigenvectors do span the whole domain  and in those cases the absence of negative eigenvalues may be sufficient for stability, e.g.\cite{KumHirMan2024}}  
	   
%
	    
	

In  the NSA case, where  we have to deal with the entire spectrum of the linearized operator,  linear stability can be established by applying the Gearhart-Pruss-Greiner (GPG) theorem \cite{EngNagBre2000} which  operates  directly with bounds on the resolvent of the linear operator. Specifically, when eigenvectors do not span the domain of the operator $A$, the GPG theorem  turns to the analysis of  another    operator 
\begin{equation}
	R_\mu=(\mu I-A)^{-1}
\end{equation}
with the parameter $\mu$ having a positive real part.  The crucial step is then to bound $R_\mu$ away from any point of the entire spectrum, not just the eigenvalues. In particular, even in infinite dimensions, such a bound rules out the cases when a sequence of eigenvalues has negative real parts converging to zero. 

%


	After the  linear stability is established, a natural step in checking the  nonlinear stability  would be,  at least in finite dimensions,  to use the  Hartman-Grobman (HG) theorem \cite{ArrPla1992}.   However,  even  in this case,  this theorem requires the absence of eigenvalues with zero real part.  Our problem has  a zero eigenvalue (a slow manifold) which  appears in the linearized operator due to translational  symmetry.   To overcome this complication,  we  use of the notion of ``stability up to shifts'' , see for instance  \cite{SafRybBer2022},   and  prove the appropriate analog of the HG theorem specifically tailored for our  infinite dimensional problem.  While there are several extensions of the HG  theorem to infinite dimensions, e.g. \cite{AulWan2000}, most of these results apply to a smooth nonlinear operators mapping a Banach space to itself whereas in our parabolic PDE problem, the operator maps a Sobolev space $H^2$ to $L^2$.  The existing HG type results for parabolic equations \cite{Lu1991}  are also not directly applicable to our problem.	Our original  approach is based on establishing subtle bounds on the derivatives of the solution  in the neighborhood of a pitchfork bifurcation which allow one to decide when the linear part of the nonlinear operator  dominates its nonlinear part. Our result is then equivalent to establishing the existence of a Lyapunov function  (or rather Lyapunov functional in our infinite dimensional setting)  for the pulse like traveling wave solutions with synchronously moving free boundaries, see   \cite{BenAbdKir2020,Mei2007,SchYan1997} for  related results.   We emphasize that our approach is readily generalizable to other  PDE models where the task is  to prove asymptotic stability  of an emerging nontrivial solution in the vicinity of a  bifurcation point.

While our approach is original, it is important to mention that a large variety of other methods for establishing nonlinear stability of traveling waves have been explored in the literature, see the reviews in  \cite{KapPro2013,pego1992eigenvalues,San2002}. In particular,  several studies  deal  specifically   with spectral stability of traveling wave solutions by showing  that the spectrum of the linearized operator   consists only of points with negative real part  \cite{LatSuk2010,ledoux2009computing,pelinovsky2005inertia}. Most of these studies use the  method of Evans function,  which  is a convenient tool for  separating the eigenvalues from the continuous spectrum ubiquitous in  traveling wave problems  defined  in unbounded domains \cite{BarHumLynLyt2018,CorJon2020,Eva1972}. 
 We do not use the Evans function based approach for two reasons. First, our traveling waves  are compact and  there  is no continuous spectrum for our problem. Second,  in our specific problem,  we can circumvent the use  of  Evans function by resorting to a   simpler  approach to calculate the leading eigenvalue developed  in \cite{CraRab1973}. Other studies of linear and nonlinear stability of traveling waves, which  use spectral theory to  obtain bounds on the semigroup generated by the linear operator and  then showing that the  nonlinear problem is dominated by the linearization,  can be found in  \cite{CorJon2020,Kap1994,KowMarMun2022}. While we basically  follow the same strategy, our  main spectral theoretic tool, which is the GPG theorem,  is different from all those used in the previous studies.
 
 No data are associated with this article.

\section{The model}\label{sec:model}

 In this Section  we briefly explain how  our  one-dimensional Keller-Segel system with free boundaries can be derived from physical considerations. To  emphasize  that this model is both minimal and universal, we present two alternative  derivations based on apparently contradicting assumptions that the material inside the cell is either  infinitely  compressibile or infinitely incompressible. 
   
In the  original, infinitely compressibile version of the model, proposed in \cite{PhysRevE.97.012410,RecPutTru2013,RecPutTru2015}, we start  by  writing the 1D force balance  for a gel segment in the form 
$$ \sigma' =\xi v,$$
where $\sigma(x,t)$ is axial stress,  $v(x,t)$ is the velocity of the gel,  $\xi$ the coefficient of viscous friction. We denote a single spacial coordinate by  $x$  and time by $t$ is time;  prime denotes the spatial  derivative.  The assumption of infinite compressibility of the gel decouples the force balance equation from the mass balance equation. Specifically, by neglecting compressibility,  we can write  the constitutive relation for an active gel, representing the material inside the cell,  in the form 
$$\sigma=\eta  v'+k m,$$
where $\eta$ is the bulk viscosity, $m(x,t)$ is the mass density of myosin motors and $k >0$ is a constant representing contractile pre-stress  per unit motor mass. 
The density of motors is  modeled by a standard  advection-diffusion equation where the advection is perceived to be  originating from  the flow of actin \cite{bois2011pattern,hawkins2011spontaneous,wolgemuth2011redundant}: 
$$\partial_t m+ (m v)'-D m''=0,$$ where $D$ is the constant effective diffusion coefficient, see \cite{RecPutTru2015} for the discussion of its  physical meaning.
Behind  this  equation is the assumption that  myosin motors, actively cross-linking the implied actin meshwork, are not only being advected by the network flow but can also diffuse due to the presence of thermal fluctuations, see also  
To ensure that the moving cell maintains its size,  we now follow \cite{barnhart2010bipedal,du2012self,loosley2012stick, recho2013asymmetry,PhysRevE.97.012410} and introduce a phenomenological cortex/osmolarity mediated quasi-elastic non-local interaction linking the front and the back of the self-propelling fragment. Specifically,  we assume  that  the boundaries of our moving active segment interact through an effective  linear spring which regulates 
the value  of the stress on the free boundaries  $l_-(t)$ and $l_+(t)$: 
$$\sigma_0^{\pm} =-k_e(L(t)-L_0)/L_0,$$ 
where $L(t)=l_{+}(t)-l_{-}(t)$ is the length of the  moving segment, $k_e$ is the effective elastic stiffness and $L_0$ is the reference length of the spring. 
Finally, we assume  that 
the free boundaries are not penetrable which means that they move with  the internal flow. Therefore the following kinemaitc boundary condition must hold: 
$\partial_t{l}_{\pm}=v(l_{\pm})$. 
Our assumptions allow us to  avoid  addressing explicitly the mass balance equation since we effectively  postulate that the addition  of  actin particles  at the front is fully compensated by  their sinchronous  removal  at the rear, see \cite{RecPutTru2013,RecPutTru2015} for details.  
We also impose a zero  flux condition for  the active component 
$m'(l_{\pm}(t),t)=0$ 
ensuring that the average concentration of motors
$m_0= L_0^{-1}\int_{l_-(t)}^{l_+(t)}m(x,t)dx$
is preserved. To complete the  setting of the ensuing  mechanical problem we  impose the initial conditions $l_{\pm}(0)=l_{\pm}^0$ and $m(x,0)=m^0(x).$  One can see that the resulting one-dimensional problem  is equivalent to a dynamical system of a Keller–Segel-type with free boundaries,  however, in contrast to the  chemotaxic analog, the nonlocality here is mechanical rather than chemical in origin.



Next, we  derive the same model under the  incompressibility assumption as it was first proposed in \cite{RybBer2023,SafRybBer2022}. This alternative derivation emphasizes the paradigmatic nature of our problem.  We  also switch now to a more formal mathematical description as appropriate for the subsequent rigorous analysis. 

The cell, whose cytoskeleton is now viewed as an incompressible fluid, is again modeled in time-dependent interval centered at a point $c(t)$ and with width $L(t)$: $\Omega(t)=(c(t)-L(t)/2,c(t)+L(t)/2)\subset\mathbb R$. For each $t\geq 0$, the velocity of fluid in the cell is $u(x,t)\in\mathbb R$ for $x\in \Omega(t)$. Since the cell is thin in the dorsal direction, most of the fluid within the cell lies close to the cell membrane. Therefore, we may assume the  flow of incompressible  fluid is dominated by friction and follows Darcy's law: $$  p' = \zeta u,$$ where $p(\cdot,t):\Omega(t)\to\mathbb R$ is the total pressure within the cell,   $\zeta>0$ is the constant adhesion coefficient and the prime denotes   spatial derivative; see  Appendix D in \cite{PhysRevE.105.064401}  for the detailed physical derivation. Following  \cite{SafRybBer2022}, we write the  equation for the total pressure  in the form $$p=\mu   u'+km,$$ where $m(\cdot,t):\Omega(t)\to\mathbb R_+$ is the density of myosin within the cell, $\mu>0$ is the viscosity coefficient, and $k>0$ is the contractility of myosin. The   incompressibility  assumption suggests that the hydrostatic fluid pressure here should be viewed as a kinematic variable which in our 1D setting is constant and can be simply absorbed  into $p$.

We further assume that again myosin density evolves in time according to the advection-diffusion equation, which, after applying Darcy's law  can be written in the form  ( see \cite{SafRybBer2022} for details):
\begin{equation} \label{myo}
	\partial_t m = Dm''-(m\phi')'.
\end{equation}
To ensure that the total myosin mass is conserved in time, the myosin density satisfies no-flux boundary conditions: $m'=0$.

A boundary condition for total pressure is obtained from an  assumption that there is a non-constitutive global elastic restoring force due to osmotic effects or the cell membrane cortex tension. This non-local effect is  modeled  by the following   condition: 
\begin{equation}
	p =-k_e(L(t)-L_0)/L_0,
\end{equation}
where $L_0$ is the length of the   cell in a  reference configuration, and $k_e>0$ is the  elastic    stiffness  coefficient. Note that in view of the possibility for the fluid to escape in  the dorsal direction,  our 1D cell is effectively compressible despite the incompressibility of the fluid,  so that  the coefficient $k_e$    can be also interpreted as  the inverse compressibility of the cell as a whole. 

The ensuing 1D problem is analytically  transparent  due to the fact that the variable $\phi=p/\zeta$ 	 satisfies  a linear  elliptic equation:
\begin{equation} \label{ell}
	- \mu \phi'' + \zeta \phi = km.
\end{equation}
The complexity of the problem resides in the boundary conditions. Thus, as we have seen,  the variable  $\phi$ at the boundary satisfies the nonlocal condition of Dirichlet type.
Furthermore, to find  the time evolution of the interval $\Omega(t)$, we assume that both components of our binary mixture, solvent (actin) and solute (miosin),  respects the  kinematic (velocity-matching) boundary conditions of Hele-Shaw type: 
\begin{align}
	\partial_tL(t)&=\phi'(c(t)+L(t)/2,t)-\phi'(c(t)-L(t)/2,t)\\
	\partial_tc(t)&=\frac{1}{2}\left(\phi'(c(t)+L(t)/2,t)+\phi'(c(t)-L(t)/2,t)\right).
\end{align}
While the resulting fluid  model, introduced  in \cite{RybBer2023,SafRybBer2022},   is mathematically  equivalent of the solid gel  model introduced in \cite{RecPutTru2013,RecPutTru2015}, it is the  Hele-Shaw fluid formulation which opened the way towards  two dimensional analysis   allowing one to track not only the position  of a moving cell but also  the evolution of its shape;  the corresponding 2D version of the gel model can be found in \cite{PhysRevLett.123.118101}.

	\section{Infinite rigidity  limit}
		Below we distinguish between two versions of our model of cell lomotion. The first   one,  to which we refer as ``Model A'' , assumes, as in the original derivation,  that the size of the cell  is controlled by an elastic spring, whose elastic modulus  will be referred to as the   ``stiffness'' parameter. By considering an asymptotic limit when such stiffness  tends to infinity, we will derive a ``stiff limit'' of the original model to which we refer in what follows as  ``Model B.'' In this limiting model, which turns out to be  analytically much more transparent,  the cell  has a fixed size.

		We recall that  in the model A,  the myosin density $m$, pressure $\phi$, length $L$, and center $c$ satisfy the PDE system:		
		\begin{numcases}{}
			-\mu\phi''+\zeta\phi=km & $x\in\Omega(t)$\label{eq:1D_pressure eqn_1}\\
			m_t=Dm''-(m\phi')' & $x\in\Omega(t)$\label{eq:1D_myosin eqn_1}\\
			\zeta\phi=-k_e\frac{L(t)-L_0}{L_0} & $x\in\partial\Omega(t)$\label{eq:1D_pressure bc_1}\\
			m'=0 & $x\in\partial\Omega(t)$\label{eq:1D_myosin bc_1}\\
			\partial_tL(t)=\phi'\left(c(t)+L(t)/2,t\right)-\phi'\left(c(t)-L(t)/2,t\right)\label{eq:1D_length_eqn} & \\
			\partial_tc(t)=\frac{1}{2}\left(\phi'\left(c(t)+L(t)/2,t\right)+\phi'\left(c(t)-L(t)/2,t\right)\right).\label{eq:1D_center_eqn} &
		\end{numcases}		
		An important feature of this model is that total myosin mass is conserved in time:
		\begin{equation}\label{eq:1D_myosin_conservation}
			\frac{d}{dt}\int_{c(t)-L(t)/2}^{c(t)+L(t)/2}m(x,t)\,dx=0.
		\end{equation}
		
		
		To nondimensionalize this model, we  rescale $x$ by $x\to x/L_0$ and accordingly  $L\to L/L_0$, $c\to c/L_0$. We then  rescale time by $t\to tD/L_0^2$, pressure by $\phi\to\phi\zeta/k_e$ and myosin density by $m\to m L_0^2/M$ where
		$M$ is the total myosin mass.
		
		After such normalization, the variables $x$, $t$, $m$, $\phi$, $L$, and $c$ are all dimensionless quantities and the PDE system \eqref{eq:1D_pressure eqn_1}-\eqref{eq:1D_center_eqn} can be re-written in the form
		\begin{numcases}{}
			-Z\phi''+\phi=Pm & $c-L/2<x<c+L/2$\label{eq:1D_phi}\\
			m_t=m''-K(m\phi')' & $c-L/2<x<c+L/2$\label{eq:1D_m}\\
			\phi=1-L & $x=c\pm L/2$\label{eq:1D_phi_BC}\\
			m_x=0 & $x=c\pm L/2$\label{eq:1D_m_BC}\\
			L_t=K(\phi'(c+L/2)-\phi'(c-L/2)) & \label{eq:1D_L}\\
			c_t=K\frac{\phi'(c+L/2)+\phi'(c-L/2)}{2}. & \label{eq:1D_c}
		\end{numcases}
Here we introduced the main non-dimensionl parameters of the model 
		\begin{equation}\label{eq:1D_parameters}
			Z=\frac{\mu}{\zeta L_0^2}\quad P=\frac{kM}{k_eL_0}\quad K=\frac{k_e}{D\zeta},
		\end{equation}
		representing dimensionless measures of bulk viscosity, activity level and the non-local stiffness, respectively.
		Note that in these rescaled coordinates, the total myosin mass is
		\begin{equation}\label{eq:1D_myosin_conservation_nondimensionalized}
			\int_{c(t)-L(t)/2}^{c(t)+L(t)/2}m(x,t)\,dx=1.
		\end{equation}

		
		A simpler and analytically  more tractable  model  can be obtained if  we consider  the limit when  the ``stiffness'' coefficient $k_e$ tends to infinity, see also \cite{RecPutTru2015}. To this end, suppose $k_e=k_e^\ast/\varepsilon$ where $\varepsilon$ is a small, positive constant. Each of the coefficients $P$, $K$ in \eqref{eq:1D_parameters} depend on $k_e$. Therefore, we denote them $P_\varepsilon=\varepsilon P_1$, and $K_\varepsilon=K_{-1}/\varepsilon$ respectively	where
		\begin{equation}
			P_1=\frac{kM}{k_e^\ast L_0},\quad K_{-1}=\frac{k_e^\ast}{D\zeta}.
		\end{equation}
For each $\varepsilon>0$, we now assume that  $\phi_\varepsilon$, $m_\varepsilon$, $L_\varepsilon$, and $c_\varepsilon$ solve \eqref{eq:1D_phi}-\eqref{eq:1D_c} with coefficients 
		 $P_\varepsilon$, and $K_\varepsilon$ for $0\leq t<T$ with initial conditions $m_\varepsilon(0,x)=\bar m(x)$, $L_\varepsilon(0)=\bar L$, and $c_\varepsilon(0)=\bar c$. We then expand each of $\phi_\varepsilon$, $m_\varepsilon$, $L_\varepsilon$, and $c_\varepsilon$ in small $\varepsilon$:
		\begin{align}
			\phi_\epsilon&=\phi_0+\varepsilon \phi_1+O(\varepsilon^2) \quad &&m_\varepsilon=m_0+\varepsilon m_1+O(\varepsilon^2)\\
			L_\varepsilon&=L_0+\varepsilon L_1+O(\varepsilon^2) \quad && c_\varepsilon=c_0+\varepsilon c_1+O(\varepsilon^2).
		\end{align}
		The next step is to substitute these expansions along with the expansions for $P_\varepsilon$ and $K_\varepsilon$ into our equations and compare terms of like power in $\varepsilon$. Note that the free boundary requires that the boundary conditions be expanded not only in $\varepsilon$, but also $x$ so that all boundary conditions are evaluated at $x=c_0\pm L_0/2$.
 
		In zero order, \eqref{eq:1D_phi} becomes $-Z_0\phi_{0,xx}+\phi_0=0$ with boundary condition $\phi_0=1-L_0$. We explicitly calculate
		\begin{equation}\label{eq:phi_0_in_stiff_limit_derivation}
			\phi_0=(1-L_0)\frac{\cosh\left(\frac{c_0-x}{\sqrt{Z_0}}\right)}{\cosh\left(\frac{L_0}{2\sqrt{Z}}\right)}
		\end{equation}
		In order $-1$, \eqref{eq:1D_m} has only one nontrivial term:
		\begin{equation}
			K_{-1}(m_0\phi_{0,x})_x=0.
		\end{equation}
		We conclude that $m_0\phi_{0,x}$ is constant in $x$. Since the initial condition $\bar m$ is arbitrary in $H^2(-L_0/2,L_0/2)$, and since $\phi_0$ given by \eqref{eq:phi_0_in_stiff_limit_derivation} does not depend on $m_0$, the this can only be accomplished if $\phi_{0,x}=0$. This, in turn, implies $\phi_0=0$ and $L_0=1$.
		
		Since $\phi_0=0$, in zero order, \eqref{eq:1D_m} becomes
		\begin{equation}
			m_{0,t}=m_{0,xx}-K_{-1}(m_0\phi_{1,x})_x.
		\end{equation}
		Where $\phi_1$ satisfies the first order expansions of \eqref{eq:1D_phi} and its boundary condition \eqref{eq:1D_phi_BC}:
		\begin{equation}
			\begin{cases}
				-Z\phi_{1,xx}+\phi_{1}=P_1m_0 & c_0-L_0/2<x<c_0+L_0/2\\
				\phi_1=-L_1 & x=c_0\pm L_0/2
			\end{cases}
		\end{equation}
		In zero order, \eqref{eq:1D_L} becomes
		\begin{equation}
			L_{0,t}=K_{-1}(\phi_{1,x}(c_0+L_0/2)-\phi_{1,x}(c_0-L_0/2)).
		\end{equation}
		On the other hand, we know that $L_0\equiv 1$, so $L_{0,t}=0$. Therefore, $\phi_1$ has to satisfy three boundary conditions:
		\begin{equation}
			\phi_1(c_0-L_0/2)=-L_1,\quad \phi_1(c_0+L_0/2)=-L_1,\quad \phi_{1,x}(c_0-L_0/2)=\phi_{1,x}(c_0+L_0/2).
		\end{equation}
		We conclude that $L_1$ is determined by these three conditions.  
		Since $\phi_1$ satisfies periodic boundary conditions, we may determine $c_0$ by the zero order expansion of \eqref{eq:1D_c}:
		\begin{equation}
			c_{0,t}=K_{-1}\phi_{1,x}(c_0+L_0/2).
		\end{equation}
				
		To complete the derivation of  our ``Model B'' , we omit the subscript indices so we write $m=m_0$,  $L=L_0=1$, $c=c_0$. It is also convenient to denote $\phi=K_{-1}\phi_1$ and introduce a new dimensionless parameter $$P=P_1K_{-1}.$$ We can then formulate the stiff limit of our   model in the form of a reduced PDE system:
		\begin{numcases}{}  
			-Z\phi_{xx}+\phi=Pm & $c(t)-1/2<x<c(t)+1/2$\label{eq:SL_phi}\\
			m_t=m_{xx}-(m\phi_x)_x & $c(t)-1/2<x<c(t)+1/2$\label{eq:SL_m}\\
			m_x=0 & $x=c(t)\pm 1/2$\label{eq:SL_m_BC}\\
			\phi(c-1/2)=\phi(c+1/2) & \label{eq:SL_phi_BC_1}\\
			\phi_x(c-1/2)=\phi_x(c+1/2) & \label{eq:SL_phi_BC_2}\\
			c_t=\phi_x(c+1/2). & \label{eq:SL_c}
		\end{numcases}
	Once again, the  total myosin mass is conserved   and its value is $1$.	Due to its relative simplicity, combined with the ability to still represent the main physical effects,   model \eqref{eq:SL_phi}-\eqref{eq:SL_c},  first introduced formally in \cite{RecPutTru2015},  will be the main focus of the present  study.
		
		 Note first that Model B   has stationary solution 
		 $$m=1,\,\,\,\phi=P.$$ As we are going to see, for $P$ sufficiently small, such  solution is asymptotically stable, but for large $P$, it becomes unstable and bifurcates into traveling wave solutions.  In the linearization of Model B about traveling waves, the invariance of traveling waves to translation manifests itself through a zero eigenvalue. By factorizing the  shifts, the corresponding eigenvector is identified with $0$, and in this way  the zero eigenvalue is effectively removed. 
		
		 The factorization of the shifts  is accomplished by changing coordinates via the transformation $x\to x-c(t)$. In the new coordinates, \eqref{eq:SL_phi}-\eqref{eq:SL_c} become
		\begin{numcases}{}
			-Z\phi_{xx}+\phi=Pm & $-1/2<x<1/2$\label{eq:SL_phi_fx}\\
			m_t=m_{xx}+\phi_x(1/2)m_x-(m\phi_x)_x & $-1/2<x<1/2$\label{eq:SL_m_fx}\\
			m_x=0 & $x=\pm1/2$\label{eq:SL_m_BC_fx}\\
			\phi(-1/2)=\phi(1/2) & \label{eq:SL_phi_BC_1_fx}\\
			\phi_x(-1/2)=\phi_x(1/2) & \label{eq:SL_phi_BC_2_fx}\\
			c_t=\phi_x(1/2). & \label{eq:SL_c_fx}
		\end{numcases}
		Observe that the center coordinate $c$ is partially decoupled from $m$. Therefore, we can drop the $c$ coordinate, effectively grouping together  all solutions that are the same up to a translation of the center    in particular , a stationary solution to \eqref{eq:SL_phi_fx}-\eqref{eq:SL_phi_BC_2_fx} becomes an element of an equivalence class of traveling wave solutions to \eqref{eq:SL_phi}-\eqref{eq:SL_c}. Furthermore, we see that if a stationary solution to \eqref{eq:SL_phi_fx}-\eqref{eq:SL_c_fx} is asymptotically stable, then the corresponding traveling wave solutions to \eqref{eq:SL_phi}-\eqref{eq:SL_c} are also asymptotically stable  ``up to shifts". Here we imply that a solution whose initial condition is a perturbation of a traveling wave solution may converge to a different traveling wave solution with the same velocity, but a different (``shifted'') center coordinate.
		
		In what follows we refer to the result of transforming coordinates in  Model B   and dropping the $c$ component as ``Model C,'' which we formulate more succinctly as:
		\begin{equation}\label{eq:model_C}
			\partial_t  m=F_C(m)=m_{xx}+\phi_x(1/2)m_x-(m\phi_x)_x,\quad \begin{cases} -Z\phi''+\phi=Pm & -1/2<x<1/2\\\phi(1/2)=\phi(-1/2)\\\phi_x(-1/2)=\phi_x(1/2).\end{cases} 
		\end{equation}
		The domain of the nonlinear operator $F_C$ is $X^2_C:=\left\{m\in H^2(-1/2,1/2):m_x(\pm 1/2)=0,\,\int_{-1/2}^{1/2}m\,dx=1\right\}.$

		\section{Stationary Solutions}\label{sec:stationary_solution_stability}
		

		
		We can now  address the stability (up to shifts) of homogeneous stationary solutions $m=1$ and $\phi=P$ of model B, by analyzing  the $m=1$ solution to model C.  Its  asymptotic stability  will be proved in the form of  the following:
		\begin{theorem}\label{thm:stability_of_SS}
			If $P/Z<\pi^2$, then there exists $\varepsilon>0$ such that if $m_0\in H^1(-1/2,1/2)$ with $\Vert m_0-1\Vert_{H^1}<\varepsilon$, then if $m(x,t)$ solves $\partial_t m=F_C(m)$ with $m(0,x)=m_0(x)$, then
			\begin{equation}
				\lim_{t\to \infty}\Vert m(\cdot,t)-m_0\Vert_{L^2}=0.
			\end{equation}
		\end{theorem}
		
		A key step in proving Theorem \ref{thm:stability_of_SS}  is the analysis of stability of the corresponding linear problem. The linearized operator $S_C$ is defined by
		\begin{equation}\label{eq:SL_linearization}
			S_Cu=DF_C(1)u=u''-\phi'',
		\end{equation}
		where $\phi$ is defined as in \eqref{eq:model_C} (with $m=u$)
		and $u\in\tilde X^2_C:=\left\{u\in H^2(-1/2,1/2):u_x(\pm 1/2)=0,\,\int_{-1/2}^{1/2}u\,dx=0\right\}$. 
		
		We   first prove the following theorem establishing the linear stability of stationary states:
		\begin{theorem}\label{thm:SS_linear_stability}
			If $P/Z<\pi^2$, then there exists $\omega>0$ such that if $m(x,t)$ solves $\partial_t u=S_Cu$ in $\tilde X_2$ with $u(0,x)=u_0(x)$, then
			\begin{equation}
				\Vert u(\cdot,t)\Vert_{L^2}<\Vert u_0\Vert_{L^2}e^{-\omega t}.
			\end{equation}
		\end{theorem}
		
		
		The  main tool in proving Theorem \ref{thm:SS_linear_stability} is expectedly the spectral theorem for compact self-adjoint operators \cite{Kre1991}, which states that a compact, self-adjoint operator has a basis of eigenvectors. The operator $S_C$ is not compact, but we will show that its inverse is, and $S_C$ shares the eigenvectors of its inverse. Therefore, we may reduce the problem of linear stability to the problem of stability of individual eigenstates, with the exponential decay in Theorem \ref{thm:SS_linear_stability} given by a uniform negativity of the corresponding eigenvalues. To complete this proof, we must only show (via three Lemmas below)  that (i) $S_C$ is self-adjoint, (ii) all eigenvalues of $S_C$ are negative and bounded away from zero, and (iii) $S_C$ has compact inverse.
		
		First we prove self-adjointness. Let $X$ be a Hilbert space and let $D(A)\subset X$ be a dense subspace of $X$ which is the domain of an operator $A:D(A)\to X$. Recall that the adjoint of  $A$ is an operator $A^\ast$ such that $\langle u,A^\ast v\rangle=\langle Au,v\rangle$ for all $u,v\in D(A)$. The operator $A$ is self-adjoint if $A=A^\ast$. Therefore, we introduce the following bilinear form to determine whether or not an operator is self-adjoint:
		\begin{definition}\label{def:adjoint_commutator}
			Let $X$ be an inner product space, and let $A:X\to X$. The \textit{adjoint commutator} of $A$ is $H:X\times X\to\mathbb R$ defined by
			\begin{equation}
				H(u,v)=\langle Au,v\rangle-\langle u,Av\rangle.
			\end{equation}
		\end{definition}
		If the adjoint commutator is identically zero, then $A$ is self-adjoint. Otherwise, $A$ is non-self-adjoint. 
		
		\begin{lemma}\label{lem:SS_self-adjoint}
			The linearization $S_C$ of $F_C$ about the stationary solution $m=1$ is self-adjoint with respect to the $L^2$ inner product.
		\end{lemma}
		\begin{proof}
			Let $H:\tilde X_2\times\tilde X_2\to\mathbb R$ be the adjoint commutator of $S_C$. Let $u_1,u_2\in\tilde X_2$. Let $\phi_i$ solve $-Z\phi_i''+\phi_i=Pu_i$ with periodic boundary conditions. Then
			\begin{align}
				H(u_1,u_2)&=\int_{-1/2}^{1/2} u_2(u_1''-\phi_1'')\,dx-\int_{-1/2}^{1/2}u_1(u_2''-\phi_2'')\,dx\\
				&=\int_{-1/2}^{1/2}(u_1'u_2'-u_2'u_1')\,dx+\int_{-1/2}^{1/2}\left[u_1\frac{\phi_2-Pu_2}{Z}-u_2\frac{\phi_1-Pu_1}{Z}\right]\,dx\\
				&=\frac{1}{PZ}\int_{-1/2}^{1/2}\left[(-Z\phi_1''+\phi_1)\phi_2-(-Z\phi_2''+\phi_2)\phi_1\right]\,dx\\
				&=\frac{1}{P}\int_{-1/2}^{1/2}(\phi_1'\phi_2'-\phi_2'\phi_1')\,dx\\
				&=0.
			\end{align}
		\end{proof}
		
		Now we show the negativity of the eigenvalues of $S_C$.	
		\begin{proposition}\label{prop:SS_eigenvalues}
			If $P/Z<\pi^2$, then all eigenvalues of $S_C$ are negative and bounded away from $0$.
		\end{proposition}
		\begin{proof}
			Assume $P/Z<\pi^2$. Let $u$ be an eigenvector of $S_C$ and let $\lambda$ be its eigenvalue. Since $S_C$ is self adjoint, $\lambda$ real and $u$ is real-valued. Without loss of generality, we may assume $\Vert u\Vert_{L^2}=1$. Let $\phi$ solve $-Z\phi''+\phi=Pu$ with periodic boundary conditions on $(-1/2,1/2)$. Then $u$ and $\lambda$ satisfy $u''-\phi''=\lambda u$. Multiplying by $u$ and integrating we find
			\begin{align*}
				\lambda&=\int_{-1/2}^{1/2}u''u-\phi''u\\
				&=-\int_{-1/2}^{1/2}u'^2\,dx-\frac{1}{Z}\int_{-1/2}^{1/2}(\phi-Pu)u\,dx\\
				&=-\int_{-1/2}^{1/2}u'^2\,dx+\frac{P}{Z}\int_{-1/2}^{1/2} u^2\,dx-\frac{1}{PZ}\int_{-1/2}^{1/2}\phi(-Z\phi''+\phi)\,dx\\
				&=-\int_{-1/2}^{1/2}u'^2\,dx+\frac{P}{Z}-\frac{1}{P}\int_{-1/2}^{1/2}\phi'^2\,dx-\frac{1}{PZ}\int_{-1/2}^{1/2}\phi^2\,dx\\
				&\leq-\int_{-1/2}^{1/2}u'^2\,dx+\frac{P}{Z}.
			\end{align*}
			It is well known that the optimal constant in the Poincar\'e inequality in an interval of length 1 is $1/\pi$ \cite{PayWei1960}. The Poincar\'e inequality applies to $u$ since $\int_{-1/2}^{1/2}u(x)\,dx=0$. Therefore, we conclude that
			\begin{equation}
				\lambda\leq \frac{P}{Z}-\pi^2<0.
			\end{equation}
			Therefore all eigenvalues of $S_C$ are negative and bounded away from $0$.
		\end{proof}

		Finally, we show that the inverse of the linearization is compact.
		\begin{lemma}\label{lem:SS_inverse_compact}
			If $P/Z<\pi^2$, then $S_C$ is invertible and $S_C^{-1}:X\to X$ is compact.
		\end{lemma}
		\begin{proof}
			Assume $P/Z<\pi^2$. By proposition \ref{prop:SS_eigenvalues}, $0$ is not an eigenvalue of $S_C$. That $S_C$ is invertible follows from the Lax-Milgram Theorem (see Proposition \ref{prop:exists_inverse} for details).
			
			First we show that $S_C^{-1}$ is bounded. Suppose, to the contrary that it is unbounded. Then there exist sequences $(v_k)\subset \tilde X^2_C$ and $(w_k)\subset L^2(-1/2,1/2)$ such that
			\begin{equation}
				S_Cv_k=w_k,\quad \Vert v_k\Vert_{L^2}=1,\quad\Vert w_k\Vert_{L^2}\leq 1/k.
			\end{equation}
			Let $\phi_k$ satisfy $-Z\phi_k''+\phi_k=Pv_k$ with periodic boundary conditions. Then the following sequence is bounded:
			\begin{equation}\label{eq:compact_resolvent_inner_product}
				\begin{split}
					\langle w_k,v_k\rangle_{L^2}&=\langle S_Cv_k,v_k\rangle_{L^2}=-\Vert v_k'\Vert_{L^2}^2+\langle \phi_k'',v_k\rangle_{L^2}.
				\end{split}
			\end{equation}
			The sequence $\langle \phi_k'',v_k\rangle_{L^2}$ is bounded in $k$ due to Proposition \ref{prop:elliptic_estimates} in Appendix A. Since the \eqref{eq:compact_resolvent_inner_product} as a whole is bounded, we conclude that $\Vert v_k'\Vert_{L_2}$ is bounded in $k$ as well.
			
			Since $\Vert v_k\Vert_{L^2}$ and $\Vert v_k'\Vert_{L^2}$ are both bounded, we conclude that $(v_k)$ is bounded with respect to the $H^1$ norm. By the Banach-Alaoglu Theorem \cite{Rud1991}, there is a subsequence also called $v_k$ which converges weakly in $H^1(-1/2,1/2)$ and thus also in $L^2$. By Morrey's inequality \cite{Dri2003} and the Arzela-Ascoli theorem \cite{Bou2013}, $H^1(-1/2,1/2)\subset\subset L^2(-1/2,1/2)$, so we may assume that $(v_k)$ converges strongly to some $v\in L^2$. Since $\Vert w_k\Vert_{L^2}\leq 1/k$, $w_k\to 0$ in $L^2(-1/2,1/2)$. Therefore, $v$ is a weak solution to $S_Cv=0$. Since $\lambda I-S_C$ is invertible, we conclude that $v=0$. But since $\Vert v_k\Vert_{L^2}=1$, we also have $\Vert v\Vert_{L^2}=1$, a contradiction. Therefore, $S_C^{-1}$ is a bounded operator.
			
			Now we show that $S_C^{-1}:L^2(-1/2,1/2)\to L^2(-1/2,1/2)$ is compact. To that end, suppose that $(v_k)\subset \tilde X^2_C$ and $(w_k)\subset L^2(-1/2,1/2)$ such that 
			\begin{equation}
				S_Cv_k=w_k\quad\text{and}\quad \Vert w_k\Vert_{L^2}\leq 1.
			\end{equation}
			We need to show that there $(v_k)$ has a convergent subsequence. But this follows from the same logic as the above step. Since $S_C^{-1}$ is a bounded operator, $(v_k)$ is a bounded sequence in $L^2(-1/2,1/2)$. Therefore, once again each term in \eqref{eq:compact_resolvent_inner_product} is bounded, so $(v_k)$ has a weakly convergent subsequence in $X^1$ and a strongly convergent subsequence in $X^0$. Therefore, $S_C^{-1}$ is compact.
		\end{proof}
		
		We can now prove the linear stability of stationary states.
		\begin{proof}[Proof of Theorem \ref{thm:SS_linear_stability}]
			Assume $P/Z<\pi^2$. Then by Proposition \ref{prop:SS_eigenvalues}, there exists $\omega>0$ so that all for all eigenvalues $\lambda$ of $S_C$, $\lambda\leq-\omega$ (in fact, we may choose $\omega=\pi^2-P/Z$). By Lemma \ref{lem:SS_inverse_compact}, $S_C$ has compact inverse. By Lemma \ref{lem:SS_self-adjoint}, $S_C$ is self-adjoint, and therefore $S_C^{-1}$ is also self-adjoint. By the spectral theorem \cite{Kre1991}, the eigenvectors of $S_C^{-1}$ form an orthogonormal set that spans a dense subset of $\tilde X_2$. Denote these eigenvectors $(v_n)$ for $n\in\mathbb N$ (since $\tilde X^2_C$ has countable dimension, we can enumerate the eigenvectors in this way). The eigenvectors of $S_C^{-1}$ are also eigenvectors of $S_C$. For each $n\in\mathbb N$, let $\lambda_n$ be the eigenvalue of $\lambda_n$ corresponding to $v_n$.
			
			Suppose $u(x,t)$ solves $\partial_t u=S_C u$ in $\tilde X_2$ with $u(0,x)=u_0(x)\in\tilde X^2_C$. Since the span of $(v_n)$ is dense in $\tilde X^2_C$, we may write any $u(x,t)$ as an infinite linear combination of these eigenvectors:
			\begin{equation}
				u(x,t)=\sum_{n=1}^\infty c_n(t)v_n(x),
			\end{equation}
			for coefficients $c_n:[0,\infty)\to\mathbb R$. Substituting this expansion into the linear evolution equation, we obtain
			\begin{equation}
				\sum_{n=1}^\infty c_n'(t)v_n(x)=\sum_{n=1}^\infty c_n(t)S_C v_n(x)=\sum_{n=1}^\infty \lambda_nc_n(t)v_n(x).
			\end{equation}
			By the orthogonality of the eigenvectors, we conclude that the sums must agree term-by-term, so for each $n$, $c_n'=\lambda_n c_n$, so
			\begin{equation}
				c_n(t)=c_n(0)e^{\lambda_n t}.
			\end{equation}
			Thus,
			\begin{align*}
				\Vert u(\cdot,t)\Vert_{L^2}&=\sqrt{\langle u,u\rangle_{L^2}}=\sqrt{\sum_{n=1}^\infty c_n^2(t)}=\sqrt{\sum_{n=1}^\infty c_n^2(0)e^{2\lambda_nt}}=e^{-\omega t}\sqrt{\sum_{n=1}^\infty c_n^2(0)}=e^{-\omega t}\Vert u(0,\cdot)\Vert_{L^2}.
			\end{align*}
			Thus, the desired result holds.
		\end{proof}
		
		
		It remains to show that the full, nonlinear stability result of Theorem \ref{thm:stability_of_SS} holds. To this end, we consider the \textit{nonlinear part} $\Psi$ of $F_C$ defined for $u\in\tilde X^2_C$  by
		\begin{equation}\label{eq:Psi}
			\Psi(u)=F_C(1+u)-S_C u=\phi'(1/2)u'-(u\phi')',
		\end{equation}
		where $\phi$ solves $-Z\phi''+\phi=Pu$ with periodic boundary conditions in $(-1/2,1/2)$. The key to our proof of nonlinear stability is showing that 
		the nonlinear part $\Psi$ dominates the linear part $S_C$. We begin with a Lemma which gives a bound on $\Psi(u)$ in terms of $u$.
		
		\begin{lemma}\label{lem:Psi_estimate}
			Let $\Psi:H^2([-1/2,1/2])\to H^1([-1/2,1/2])$ be defined as in \eqref{eq:Psi}. Then there exists $C>0$ independent of $u$, $P$, and $Z$ such that
			\begin{equation}
				\Vert\Psi(u)\Vert_{L^2}\leq \frac{CP}{Z}\Vert u\Vert_{L^2}\Vert u\Vert_{H^1}.
			\end{equation}
		\end{lemma}
		\begin{proof}
			We make a direct calculation using estimates from Proposition \ref{prop:elliptic_estimates} in Appendix A:
			\begin{align}
				\Vert\Psi(u)\Vert_{L^2}^2&=\int_{-1/2}^{1/2}[(\phi'(1/2)-\phi'(x))u'(x)-u(x)\phi''(x)]^2\,dx\\
				&\leq2\int_{-1/2}^{1/2}(\phi'(1/2)-\phi'(x))^2(u'(x))^2\,dx+2\int_{-1/2}^{1/2}(\phi''(x))^2(u(x))^2\,dx\\
				&\leq 4|\phi'(1/2)|^2\Vert u'\Vert_{L^2}^2+4\Vert\phi'^2u'^2\Vert_{L^1}+2\Vert\phi''^2u^2\Vert_{L^1}\\
				&\leq \frac{P^2}{Z^2}\Vert u\Vert_{L^1}^2\Vert u'\Vert_{L^2}^2+4\Vert\phi'\Vert_{L^\infty}^2\Vert u'\Vert_{L^2}^2+2\Vert\phi''\Vert_{L^4}^2\Vert u\Vert_{L^4}^2\\
				&\leq \frac{2P^2}{Z^2}\Vert u\Vert_{L^1}^2\Vert u\Vert_{H^1}^2+\frac{8P^2}{Z^2}\Vert u\Vert_{L^4}^4\label{eq:Psi_pre}
			\end{align}
			By H\"older's inequality, $\Vert u\Vert_{L^1}\leq \Vert u\Vert_{L^2}$. From the Gagliardo-Nirenberg inequality, there exists $C_1>0$ independent of $u$ such that
			\begin{equation}\label{eq:GN_inequality}
				\Vert u\Vert_{L^4}\leq C_1\Vert u\Vert_{H^1}^{1/2}\Vert u\Vert_{L^1}^{1/2}.
			\end{equation}
			Substituting \eqref{eq:GN_inequality} into \eqref{eq:Psi_pre} and letting $C^2=2+8C_1$, we obtain
			\begin{equation}
				\Vert\Psi(u)\Vert_{L^2}\leq C\frac{P}{Z}\Vert u\Vert_{L^2}\Vert u\Vert_{H^1}.
			\end{equation}
		\end{proof}
				
		Our next goal is to show that if $\Vert u\Vert_{L^2}$ is small, then so is $\Vert \Psi(u)\Vert_{L^2}$. However, Lemma \ref{lem:Psi_estimate} is not sufficient to accomplish this because even if $\Vert u\Vert_{L^2}$ is small, $\Vert u\Vert_{H^1}$ may be large. Therefore, the following Lemma shows that if $\Vert u(\cdot,t)\Vert_{L^2}$ is small for all $t$, then $\Vert u'(\cdot,t)\Vert_{L^2}$ does not exceed $\Vert u'(\cdot,0)\Vert_{L^2}$. 
		\begin{lemma}\label{lem:w_derivative_bounded}
			Suppose $P/Z<\pi^2$. Let $T,\varepsilon>0$ and let $u$ be a solution to $\partial_tu=S_Cu+\Psi(u)$ in $C^1([0,T];\tilde X^2_C)$. There exists $U^\ast>0$ such that if $\Vert u'(\cdot,0)\Vert_{L^2}<\varepsilon$ and $\Vert u(\cdot,t)\Vert_{L^2}<U^\ast$ for all $0\leq t\leq T$, then
			\begin{equation}
				\Vert u'(\cdot,t)\Vert_{L^2}\leq \varepsilon\quad \text{for all $0\leq t\leq T$.}
			\end{equation}
		\end{lemma}
		\begin{proof}
			Write the evolution equation for $u$ as
			\begin{equation}
				\partial_t  u-u''=-\phi''+\Psi(u).
			\end{equation}
			Square both sides and integrate to obtain
			\begin{align}
				\Vert -\phi''+\Psi(u)\Vert_{L^2}^2&=\int_{-1/2}^{1/2} (\partial_t  u)^2-2(\partial_t  u)u''+(u'')^2\,dx\\
				&=\Vert\partial_t  u\Vert_{L^2}^2+2\int_{-1/2}^{1/2}(\partial_t  u')u'\,dx+\Vert u''\Vert_{L^2}^2\\
				&=\Vert\partial_t  u\Vert_{L^2}^2+\frac{d}{dt}\Vert u'\Vert_{L^2}^2+\Vert u''\Vert_{L^2}^2.
			\end{align}
			Thus,
			\begin{align}
				\frac{d}{dt}\Vert u'\Vert^2_{L^2}&\leq\Vert -\phi''+\Psi(u)\Vert_{L^2}^2-\Vert u''\Vert_{L^2}^2\\
				&\leq2\Vert \phi''\Vert_{L^2}^2+2\Vert\Psi(u)\Vert_{L^2}^2-\Vert u''\Vert_{L^2}^2.
			\end{align}
			From Lemma \ref{lem:Psi_estimate}, there exists $C_1$ independent of $u$ such that $\Vert\Psi(u)\Vert_{L^2}\leq C_1\Vert u\Vert_{L^2}\Vert u\Vert_{H^1}.$ Moreover, by Proposition \ref{prop:elliptic_estimates} in Appendix A, $\Vert\phi''\Vert_{L^2}\leq 2P/Z\Vert u\Vert_{L^2}$. Since $\int_{-1/2}^{1/2}u\,dx=0$ and $u'(\pm 1/2,t)=0$, we may apply the Poincar\'e inequality to both $u$ and $u'$ with the optimal Poincar\'e constant $1/\pi$:
			\begin{equation}\label{eq:Poincare_m_prime}
				\Vert u\Vert_{L^2}\leq\frac{1}{\pi}\Vert u'\Vert_{L^2}\quad\text{and}\quad\Vert u'\Vert_{L^2}\leq\frac{1}{\pi}\Vert u''\Vert_{L^2}.
			\end{equation}
			Thus,
			\begin{align}
				\frac{d}{dt}\Vert u'\Vert^2_{L^2}&\leq 8\frac{P^2}{Z^2}\Vert u\Vert_{L^2}^2+2C_1^2\Vert u\Vert_{L^2}^2\Vert u\Vert_{H^1}^2-\pi^2\Vert u'\Vert_{L^2}^2\\
				&\leq 8\frac{P^2}{Z^2}\Vert u\Vert_{L^2}^2+2C_1^2\Vert u\Vert_{L^2}^2(\Vert u\Vert_{L^2}+\Vert u'\Vert_{L^2})^2-\pi^2\Vert u'\Vert_{L^2}^2\\
				&\leq-\left(\pi^2-4C_1^2\Vert u\Vert_{L^2}^2\right)\Vert u'\Vert_{L^2}^2+8\frac{P^2}{Z^2}\Vert w\Vert_{L^2}^2.
			\end{align}
			Let $U^\ast<\pi/(4C_1)$. Then if $\Vert u\Vert_{L^2}\leq U^\ast$ for all $0<t<T$,
			\begin{equation}
				\frac{d}{dt}\Vert u'\Vert^2_{L^2}\leq -R_1\Vert u'\Vert_{L^2}^2+R_2
			\end{equation}
			where
			\begin{equation}
				R_1=\pi^2-4C_1^2(U^\ast)^2>\frac{\pi^2}{2}\quad\text{and}\quad R_2=8\frac{P^2}{Z^2}(U^\ast)^2.
			\end{equation}
			
			Let $q(t)=\Vert u'(\cdot,t)\Vert_{L^2}^2-R_2/R_1$. Then $q$ satisfies $q'\leq -R_1q.$ By the Gr\"onwall's inequality, $q(t)\leq q(0)e^{-R_1t}.$	We conclude that if $q(0)<0$, then $q(t)<0$ for all $t>0$. Thus, if $\Vert u\Vert_{L^2}\leq U^\ast$, and if
			\begin{equation}
				\Vert u'(\cdot,0)\Vert_{L^2}<\sqrt{\frac{R_2}{R_1}},\quad \text{then}\quad
				\Vert u'(\cdot,t)\Vert_{L^2}<\sqrt{\frac{R_2}{R_1}}
			\end{equation}
			for all $t>0$. Letting $U^\ast<\varepsilon/(4\pi)$, we have
			\begin{equation}
				\sqrt{\frac{R_2}{R_1}}<4\pi U^\ast<\varepsilon,
			\end{equation}
			so the desired result holds.
		\end{proof}
		
		\begin{proof}[Proof of Theorem \ref{thm:stability_of_SS}]
			Let $u=m-1$. Then $u$ solves
			\begin{equation}\label{eq:u_eqn}
				\begin{cases}
					\partial_t  u=S_Cu+\Psi(u) & -1/2<x<1/2,\;t>0\\
					u'=0 & x=\pm 1/2,\;t>0\\
					u(\cdot,0)=m_0-1:=u_0 & t=0.
				\end{cases}
			\end{equation}
			Let $S(t)$ be the semigroup generated by $S_C$. By Theorem \ref{thm:SS_linear_stability}, there exists $\omega>0$ such that $\Vert S(t)\Vert_{L^2}\leq e^{-\omega t}$. We observe that
			\begin{align}
				\int_0^t S(t-\tau)\Psi(u(\tau,\cdot))\,d\tau&=\int_0^t S(t-\tau)(u'(\tau)-S_Cu(\tau,\cdot))\,d\tau\\
				&=\int_0^t\frac{d}{d\tau}\left(S(t-\tau)u(\tau,\cdot)\right)\,d\tau\\
				&=S(t-\tau)u(\tau,\cdot)\Big|_{0}^t\\
				&=u(\cdot,t)-S(t)u(0,\cdot)\\
				&=u(\cdot,t)-S(t)u_0.
			\end{align}
			Thus we may write
			\begin{equation}\label{eq:SS_u_integral_eqn}
				u(\cdot,t)=S(t)u_0+\int_0^t S(t-\tau)\Psi(u(\tau,\cdot))\,d\tau.
			\end{equation}
			
			Taking the $L^2$ norm of both sides, we find that
			\begin{equation}
				\Vert u(\cdot,t)\Vert_{L^2}\leq e^{-\omega t}\Vert u_0\Vert_{L^2}+\int_0^t e^{\omega(\tau-t)}\Vert \Psi(u(\tau,\cdot))\Vert_{L^2}\,d\tau.
			\end{equation}
			Lemma \ref{lem:Psi_estimate} provides as estimate for $\Psi$ in terms of a constant $C$, leading to
			\begin{equation}
				\Vert u(\cdot,t)\Vert_{L^2}\leq e^{-\sigma t}\Vert u_0\Vert_{L^2}+ C\int_0^t e^{\omega(\tau-t)}\Vert u(\tau,\cdot)\Vert_{L^2}\Vert u(\tau,\cdot)\Vert_{H^1}\,d\tau.
			\end{equation}
			
			Let $\varepsilon=\omega\pi/(2 C(1+\pi))$, and let $U^\ast$ be as in Lemma \ref{lem:w_derivative_bounded}. Suppose that $\Vert u'(0,\cdot)\Vert_{L^2}<\varepsilon$ and $\Vert u(0,\cdot)\Vert_{L^2}<U^\ast$. Let
			\begin{equation}
				W=\{t\geq 0: \Vert u(\tau,\cdot)\Vert_{L^2}\leq U^\ast\text{ for all $0\leq \tau\leq t$}\}.
			\end{equation}
			By continuity, $W$ is a closed interval and $0\in W$. Thus, either $W=[0,\infty)$ or $W$ has a positive maximum. Let $T\in W$. Then, after applying the Poincar\'e inequality and Lemma \ref{lem:w_derivative_bounded}, for any $0\leq t\leq T$,
			\begin{equation}
				\begin{split}
					\Vert u(\cdot,t)\Vert_{L^2}&\leq e^{-\omega t}\Vert u_0\Vert_{L^2}+C\int_0^t e^{\omega(\tau-t)}\Vert u(\tau,\cdot)\Vert_{L^2}\left(1+\frac{1}{\pi}\right)\Vert u'(\tau,\cdot)\Vert_{L^2}\,d\tau\\
					&\leq e^{-\omega t}\Vert u_0\Vert_{L^2}+C\int_0^t e^{\omega(\tau-t)}\Vert u(\tau,\cdot)\Vert_{L^2}\left(1+\frac{1}{\pi}\right)\varepsilon\,d\tau\\
					&\leq e^{-\omega t}\Vert u_0\Vert_{L^2}+\frac{\omega}{2}\int_0^t e^{\omega(\tau-t)}\Vert u(\tau,\cdot)\Vert_{L^2}\,d\tau.
				\end{split}
			\end{equation}
			Therefore, by Gr\"onwall's inequality,
			\begin{equation}\label{eq:SS_bound_on_u}
				\Vert u(\cdot,t)\Vert_{L^2}\leq \Vert u_0\Vert_{L^2}e^{-\omega t/2}
			\end{equation}
			for all $0\leq t\leq T$. Therefore, $\Vert u(\cdot,T)\Vert_{L^2}< U^\ast$ and so by continuity, $T\neq \max W$. Since $T\in W$ is arbitrary, we conclude that $W$ does not have a maximum so $W=[0,\infty)$ and  $\Vert u(\cdot,t)\Vert_{L^2}\leq U^\ast$ for all $t>0$. That is, \eqref{eq:SS_bound_on_u} holds for all $t\geq 0$ provided $\Vert u(0,\cdot)\Vert_{L^2}<U^\ast$ and $\Vert u'(0,\cdot)\Vert_{L^2}<\varepsilon$. Note that from the proof of Lemma \ref{lem:w_derivative_bounded}, $U^\ast<\varepsilon$. Therefore, the desired result holds.
		\end{proof}
		
	\section{Traveling Waves}
		
		By now we have seen that provided $P/Z<\pi^2$, the stationary solution $m=1$ to model  B  is asymptotically stable (up to shifts). 
		A key step in the proof of this stability result was Proposition \ref{prop:SS_eigenvalues}, which shows that for $P/Z<\pi^2$, all eigenvalues of $S_C$ are negative. 
		
		First observe that $u(x)=\cos(2\pi x)$ is an eigenvector of $S_C$ (since it is also an eigenvector of $-Z\phi''+\phi$ with periodic boundary conditions on $(-1/2,1/2)$). Its eigenvalue is $-4\pi^2+\frac{P}{Z}\frac{1}{1+\frac{1}{4\pi ^2Z}}$. If, for fixed $Z>0$,  the parameter  $P$ is large enough,  this eigenvalue is positive. This observation hints  that for some critical value $P_0>\pi^2 Z$ of $P$, the largest eigenvalue of $S_C$ will reach  zero, and for $P>P_0$, the stationary solution $m=1$ will become  unstable.
		
		In this section, we show that for any $Z>0$, there exists a number $V^\ast>0$  and a smooth function $P_T:(-V^\ast,V^\ast)\to\mathbb R$ such that if $P=P_T(V)$, then there exists a traveling wave solution to the model B with velocity $V$ and center $c=Vt$. This family of traveling wave solutions parameterized by $V$ bifurcates from the family of stationary states  at $V=0$ and $P=P_T(0)=P_0$. As we show, this bifurcation is of the type   illustrated in Figure \ref{fig:bifurcation_diagram}.
		\begin{figure}
			\centering
f			\includegraphics[width=6cm]{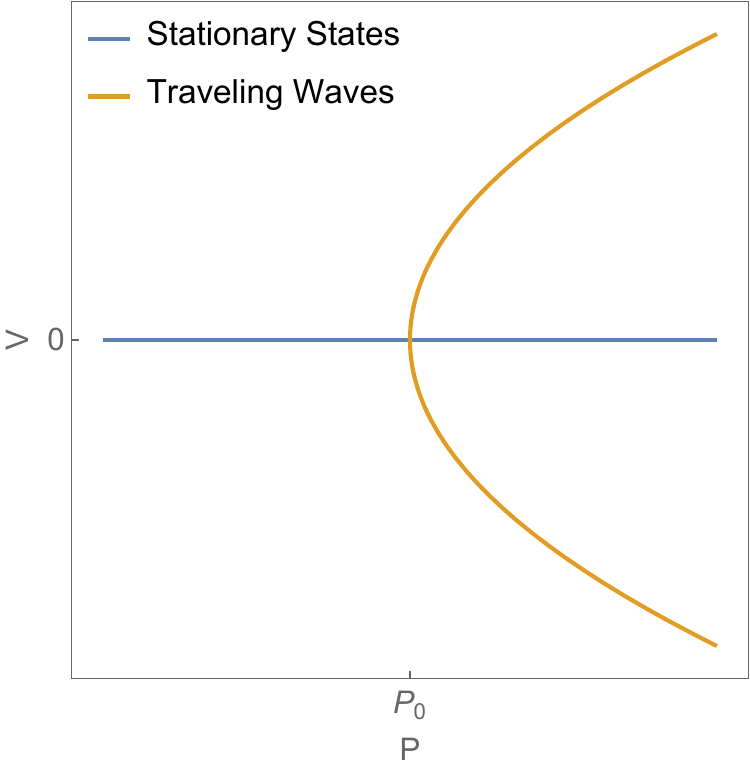}
			\caption{This diagram shows the pitchfork bifurcation from stationary states to traveling waves which is structurally the same  in our all   three models.}
			\label{fig:bifurcation_diagram}
		\end{figure}
		
	\begin{remark}
		We note that the  presence of the implied bifurcation was first discovered in   \cite{RecPutTru2013,RecPutTru2015} where  the  structure of the associated  bifurcation for both model A and model B was identified as well by using formal expansions. Here, in the framework of model B, we complement this earlier study not only  by a rigorous analysis of the existence of the  traveling wave solutions but also by   providing  a  global analysis of the corresponding bifurcation.
		\end{remark}
		
		Observe first that $m(x,t)$ is a traveling wave solution with velocity $V$ in model B if $m(x,t)=m_T(x-Vt)$ where $m_T$ satisfies
		\begin{equation}\label{eq:TW_equations}
			\begin{cases}
				m_T''+Vm_T'-(m_T\phi_T')'=0 & -1/2<x<1/2\\
				m_T'(\pm 1/2)=0 & 
			\end{cases}
			\quad\text{with}\quad
			\begin{cases}
				-Z\phi_T''+\phi_T'=P(V)m_T & -1/2<x<1/2\\
				\phi_T(1/2)=\phi_T(-1/2) & \\
				\phi_T'(1/2)=\phi_T'(-1/2)=V. &
			\end{cases}
		\end{equation} 
		We further observe that a solution to this equation is $m_T=\Lambda e^{\phi_T-V x}$ for any $\Lambda\in\mathbb R$. The value of $\Lambda=\Lambda(V)$ can be determined by the provision that $\int_{-1/2}^{1/2}m\,dx=1$, and the value of $\phi_T$ therefore satisfies
		\begin{equation}\label{eq:TW_phi}
			\begin{cases}
				-Z\phi_T''+\phi_T=P(V)\Lambda(V)e^{\phi_T(x)-Vx} & -1/2<x<1/2\\
				\phi_T(-1/2)=\phi_T(1/2) & \\
				\phi_T'(-1/2)=\phi_T'(1/2)=V. &
			\end{cases}
		\end{equation}
		Note that \eqref{eq:TW_phi} has three boundary conditions: not only must $\phi_T$ satisfy periodic boundary conditions, but also $\phi_T'(\pm 1/2)=V$. Thus, $P(V)$ is selected so that $\phi_T$ can satisfy this extra condition.
		
		Solutions to \eqref{eq:TW_phi} may be approximated asymptotically in small $V$ as shown in the following Lemma:
		\begin{lemma}\label{lem:TW_asymptotic_form}
			Let $Z>0$. Suppose that $m_T$, $\phi_T$ and $P(V)$ solve \eqref{eq:TW_equations}. In small $V$, $P(V)$ and $m_T$ have the asymptotic forms
			\begin{align}
				P(V)&=P_0+V^2 P_2+O(V^4)\label{eq:P_expansion}\\
				m_T(x)&=1+Vm_1(x)+V^2m_2(x)+O(V^3),\label{eq:m_expansion}
			\end{align}
			where
			\begin{itemize}
				\item $P_0$ solves
				\begin{equation}\label{eq:P0}
					\tanh\left(\frac{\sqrt{1-P_0}}{2\sqrt{Z}}\right)=P_0\frac{\sqrt{1-P_0}}{2\sqrt{Z}}\quad\text{or equivalently}\quad\tan\left(\frac{\sqrt{P_0-1}}{2\sqrt{Z}}\right)=P_0\frac{\sqrt{P_0-1}}{2\sqrt{Z}}.
				\end{equation}
				\item $P_2$ is given by
				\begin{equation}\label{eq:P2}
					P_2=\frac{P_0 \left(6 P_0^6-15 P_0^5-3 P_0^4 (56 Z-5)+P_0^3 (514 Z-6)-1044 P_0^2 Z+72 P_0 Z (55 Z-1)+5280 Z^2\right)}{288 (P_0-1)^4 \left(P_0^2-12 Z\right)}.
				\end{equation}
				\item $m_1$ is given by
				\begin{equation}\label{eq:m1}
					m_1(x)=\frac{x}{P_0-1}-\frac{P_0 \sqrt{\frac{P_0^3-P_0^2+4 Z}{\left(P_0-1\right) P_0^2}} \sin \left(\frac{\sqrt{P_0-1} x}{\sqrt{Z}}\right)}{2 P_0-2}
				\end{equation}
				\item $m_2$ is given by
				\begin{equation}\label{eq:m2}
					m_2(x)=A+Bx^2+(C+Dx^2)\cos\left(\frac{\sqrt{P_0-1}}{\sqrt{Z}}x\right)+Ex\sin\left(\frac{\sqrt{P_0-1}}{\sqrt{Z}}x\right)+F\cos\left(\frac{2\sqrt{P_0-1}}{\sqrt{Z}}x\right)
				\end{equation}
				with
				\begin{align}
					A&=\frac{-6 P_0 Z+P_0+24 Z-1}{24 \left(P_0-1\right){}^4}\\
					B&=\frac{12-12 P_0}{24 \left(P_0-1\right){}^4}\\
					C&=\frac{\left(P_0 (28 Z-3)+3 P_0^2-60 Z\right) \sqrt{\frac{P_0^3-P_0^2+4 Z}{Z}}}{96 \left(P_0-1\right){}^4}\\
					D&=-\frac{P_0 \sqrt{\frac{P_0^3-P_0^2+4 Z}{Z}}}{8 \left(P_0-1\right){}^3}\\
					E&=\frac{\left(4-3 P_0\right) \sqrt{Z} \sqrt{\frac{P_0^3-P_0^2+4 Z}{Z}}}{8 \left(P_0-1\right){}^{7/2}}\\
					F&=\frac{\left(3-4 P_0\right) \left(P_0^3-P_0^2+4 Z\right)}{48 \left(P_0-1\right){}^4}.
				\end{align}
			\end{itemize}
		\end{lemma}
		
		\begin{proof}
			We first introduce expansions for $\Lambda(V)$ and $\phi_T$:
			\begin{align}
				\Lambda(V)&=\Lambda_0+V^2\Lambda_2 +O(V^4),\label{eq:Lambda_expansion}\\
				\phi_T(x)&=\phi_0(x)+V\phi_1(x)+V^2\phi_2(X)+V^3\phi_3(x)+O(V^4).\label{eq:phi_expansion}
			\end{align}
			Observe that neither the expansion for $P$ \eqref{eq:P_expansion} nor the expansion for $\Lambda$ \eqref{eq:Lambda_expansion} have terms that are of odd order in $V$. This is because we expect symmetry in traveling wave solutions with respect to the sign of $V$. Therefore, $P(V)$ and $\Lambda(V)$ are even functions of $V$ and for $m_T(x)$ and $\phi_T(x)$, the transformation $V\mapsto -V$ is equivalent to $x\mapsto -x$.			
			
			Since $m_T(x)=\Lambda(V)e^{\phi_T(x)-V x}$, we have may expand the exponential to obtain $1=\Lambda_0e^{\phi_0}$,  $m_1(x)=\Lambda_0e^{\phi_0}(\phi_1-x)$ and
			\begin{equation}\label{eq:phi2_to_m2}
				m_2(x)=\Lambda _2 e^{\phi _0}+\frac{1}{2} \Lambda _0 e^{\phi _0} \left(x^2-2 x \phi _1+\phi _1^2+2 \text{$\phi $2}\right).
			\end{equation}
			We conclude that $\phi_0$ is constant and $\Lambda_0=e^{-\phi_0}$. Substituting the expansion \eqref{eq:phi_expansion} into \eqref{eq:TW_phi} and comparing terms of like order in $V$, we obtain $\phi_0=P_0$ in zero order and the following differential equations in first through third order:
			\begin{equation}\label{eq:first_order}
				\begin{cases}
					-Z\phi_1''+(1-P_0)\phi_1=-P_0x & -1/2<x<1/2\\
					\phi_1(1/2)=\phi_1(-1/2) & \\
					\phi_1'(1/2)=\phi_1'(-1/2)=1
				\end{cases}
			\end{equation}
			
			\begin{equation}\label{eq:second_order}
				\begin{cases}
					-Z\phi_2''+(1-P_0)\phi_2=P_2+P_0e^{P_0}\Lambda_2+\frac{P_0}{2}(\phi_1-x)^2 & -1/2<x<1/2\\
					\phi_2(1/2)=\phi_2(-1/2) & \\
					\phi_2'(1/2)=\phi_2'(-1/2)=0
				\end{cases}
			\end{equation}
			
			\begin{equation}\label{eq:third_order}
				\begin{cases}
					-Z\phi_3''+(1-P_0)\phi_3=\left(\phi _1-x\right) \left(P_0 \left(\Lambda _2 e^{P_0}+\phi _2\right)+P_2\right)+\frac{1}{6} P_0 \left(\phi _1-x\right){}^3 & -1/2<x<1/2\\
					\phi_2(1/2)=\phi_2(-1/2) & \\
					\phi_2'(1/2)=\phi_2'(-1/2)=0.
				\end{cases}
			\end{equation}
			
			The solution to \eqref{eq:first_order} is
			\begin{equation}\label{eq:phi1}
				\phi_1(x)=\frac{P_0x}{P_0-1}-\frac{1}{2}\frac{P_0}{P_0-1}\csc \left(\frac{\sqrt{P_0-1}}{2\sqrt{Z}}\right) \sin \left(\frac{\sqrt{P_0-1}}{2\sqrt{Z}}x\right).
			\end{equation}
			In order to satisfy the additional condition $\phi_1'(\pm 1/2)=1$, $P_0$ must solve \eqref{eq:P0}. Therefore, \eqref{eq:m1} is obtained as $m_1(x)=\phi_1(x)-x$.
			
			In \eqref{eq:second_order}, $\Lambda_2$ is determined by the condition that
			\begin{equation}
				\frac{d^2}{dV^2}\int_{-1/2}^{1/2}m_T(x)\,dx=\frac{d^2}{dV^2}\int_{-1/2}^{1/2}\Lambda(V)e^{\phi_T-Vx}\,dx=0.
			\end{equation}
			The value of $\Lambda_2$ is
			\begin{equation}
				\Lambda_2=-e^{-P_0}P_2-\frac{e^{-P_0} \left(3 P_0^2-60 Z+2\right)}{48 \left(P_0-1\right){}^2}.
			\end{equation}
			The solution to \eqref{eq:second_order} can be found using elementary methods and has the form
			\begin{equation}\label{eq:phi2}
				\phi_2(x)=P_2+a_0+a_2x^2+(b_0+b_2x^2)\cos\left(\frac{\sqrt{P_0-1}}{\sqrt{Z}}x\right)+c_1x\sin\left(\frac{\sqrt{P_0-1}}{\sqrt{Z}}x\right)+d_0\cos\left(\frac{2\sqrt{P_0-1}}{\sqrt{Z}}x\right)
			\end{equation}
			where
			\begin{align}
				a_0&=\frac{P_0^2 (1-30 Z)+P_0 (48 Z-1)}{24 \left(P_0-1\right){}^4}\\
				a_2&=-\frac{P_0}{2 \left(P_0-1\right){}^3}\\
				b_0&=\frac{\left(P_0 (28 Z-3)+3 P_0^2-60 Z\right) \sqrt{P_0^3-P_0^2+4 Z}}{96 \left(P_0-1\right){}^4 \sqrt{Z}}\\
				b_2&=-\frac{P_0 \sqrt{P_0^3-P_0^2+4 Z}}{8 \left(P_0-1\right){}^3 \sqrt{Z}}\\
				c_1&=\frac{P_0 \sqrt{P_0^3-P_0^2+4 Z}}{8 \left(P_0-1\right){}^{7/2}}\\
				d_0&=\frac{-4 P_0 Z-P_0^4+P_0^3}{48 \left(P_0-1\right){}^4}.
			\end{align}
			Note that the only dependence on $P_2$ in $\phi_2$ is the leading term---none of the other coefficients depend on $P_2$. Therefore, we write $\phi_2=P_2+\tilde\phi_2$, where $\tilde\phi_2$ is independent of $P_2$. We similarly write $\Lambda_2=-e^{-P_0}P_2+\tilde\Lambda_2$.
			
			In third order, we do not need to find an explicit solution $\phi_3$. Instead, we divide the right hand side of the differential equation in \eqref{eq:third_order} to separate terms that explicitly depend on $P_2$ form those that do not:
			\begin{equation}
				-Z\phi''_3+(1-P_0)\phi_3=P_2f(x)+g(x),
			\end{equation}
			where $f(x)=\phi_1-x$ and $g(x)=P_0 \left(\phi _1-x\right) \left(e^{P_0} \tilde\Lambda_2+\tilde{\phi}_2\right)+\frac{1}{6} P_0 \left(\phi _1-x\right){}^3$. The three boundary conditions that $\phi_3$ must satisfy (periodic boundary conditions with $\phi_3'(\pm 1/2)=0$) determine $P_2$, which we show as follows. Let $U(x)=\sin\left(x\sqrt{P_0-1}/\sqrt{Z}\right)$. Then
			\begin{align}
				\int_{-1/2}^{1/2}(P_2f(x)+g(x))U(x)\,dx&=\int_{-1/2}^{1/2}(-Z\phi_3''+(1-P_0)\phi_3)U(x)\,dx\\
				&=-Z\phi_3'U(x)\big|_{-1/2}^{1/2}+Z\phi_3U'\big|_{-1/2}^{1/2}+\int_{-1/2}^{1/2}\phi_3(-ZU''+(1-P_0)U)\,dx\\
				&=0.
			\end{align}
			Therefore, we may explicitly calculate
			\begin{equation}\label{eq:P2_calc}
				\begin{split}
					P_2&=-\frac{\int_{-1/2}^{1/2}g(x)U(x)\,dx}{\int_{-1/2}^{1/2}f(x)U(x)\,dx}\\
					&=\frac{P_0 \left(6 P_0^6-15 P_0^5-3 P_0^4 (56 Z-5)+P_0^3 (514 Z-6)-1044 P_0^2 Z+72 P_0 Z (55 Z-1)+5280 Z^2\right)}{288 (P_0-1)^4 \left(P_0^2-12 Z\right)}.
				\end{split}
			\end{equation}
			Therefore, substituting \eqref{eq:P2_calc} into \eqref{eq:phi2}, and then into \eqref{eq:phi2_to_m2}, we obtain \eqref{eq:m2}.			
		\end{proof}
		\begin{remark}
			The existence of this bifurcation, demonstrated for the 1D problem in \cite{RecPutTru2013, RecPutTru2015}, was shown for the 2D problem in \cite{RybBer2023} (see also \cite{SafRybBer2022}). The value of $P_2$ in \eqref{eq:P2} is identical to the result given in \cite{RecPutTru2015} (see Appendix D).  
		\end{remark}
		\begin{remark}
			Lemma \ref{lem:TW_asymptotic_form} gives the asymptotic form of traveling wave solutions if they exist, but it does not prove that such solutions exist. To prove existence, we have Theorem \ref{thm:bifurcation} below.
		\end{remark}
		
		A plot of the asymptotic approximation of $m_T(x)$ given by Lemma \ref{lem:TW_asymptotic_form} for several values of $V$ is given in Figure \ref{fig:TW_plot}.
		\begin{figure}
			\centering
			\includegraphics[width=3in]{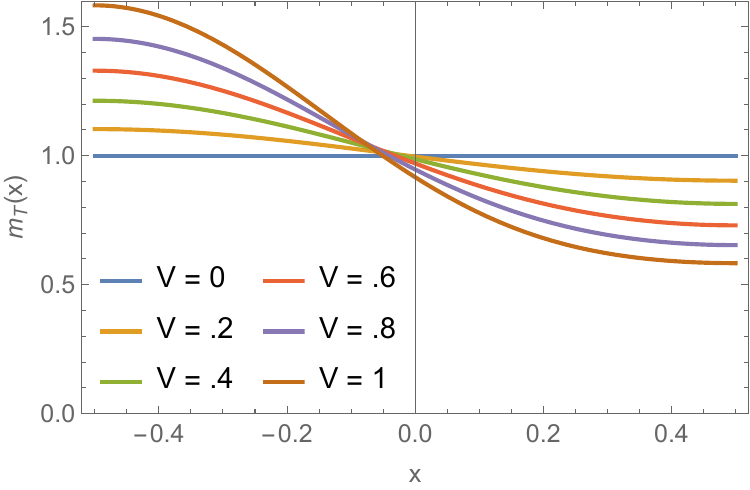}
			\caption{Traveling waves with low velocity have nearly constant myosin density ($m_T\approx 1$), but traveling waves with higher velocity are increasingly asymmetric.}
			\label{fig:TW_plot}
		\end{figure}

		
		 We can now prove  the following  result (see also   \cite{RecPutTru2015}):
		\begin{theorem}\label{thm:bifurcation}
			Let $Z>0$ and suppose $P_0$ satisfies
			\begin{equation}\label{eq:bifurcation_condition}
				\tanh\left(\frac{\sqrt{1-P_0}}{2\sqrt{Z}}\right)=P_0\frac{\sqrt{1-P_0}}{2\sqrt{Z}}.
			\end{equation}
			and $\frac{1-P_0}{Z}\neq -n^2\pi^2$ for any integer $n$. Then there exists $V^\ast>0$ and a continuous function $P_T:(-V^\ast,V^\ast)\to\mathbb R$ such that for each $V\in(-V^\ast,V^\ast)$, there exists a family traveling wave solutions $(m_T,\phi_T,Vt+c_0)$ of velocity $V$ to \eqref{eq:SL_phi}-\eqref{eq:SL_c} with $P=P_T(V)$ and $c_0\in\mathbb R$. Moreover, $P_T(0)=P_0$.
		\end{theorem}
		The parameter $V^\ast$ in Theorem \ref{thm:bifurcation} is the (not explicitly known) largest velocity for which traveling waves must exist. That is, the bifurcation of stationary solutions to traveling waves is a strictly local result in a neighborhood of $(m,V)=(1,0)$. The main tool to prove this theorem will the be the CR theorem \cite{CraRab1971}, which we quote in Appendix B. 
		
		Essentially, the CR theorem gives conditions under which an equation of the form $\mathcal F(x,t)=0$ has two families of solutions: a \textit{trivial branch} where $x=0$ and $t$ parameterizes the family, and a \textit{nontrivial branch} where $x$ and $t$ are both parameterized by a new parameter $s$, and the two families meet at $(x,t)=(0,0)$. In Theorem \ref{thm:bifurcation}, the  trivial branch corresponds to the stationary homogeneous solution $m=1$ for any value of $P$. The nontrivial branch corresponds to the traveling wave solutions parameterized by their velocity and with activity parameter $P=P_T(V)$. The two families of solutions meet at $P=P_T(0)=P_0$ satisfying \eqref{eq:bifurcation_condition}.
		\begin{proof}[Proof of Theorem \ref{thm:bifurcation}]
			Given $Z>0$, let $P=P_0$ be such that \eqref{eq:bifurcation_condition} holds. Let
			\begin{equation}
				X=\left\{\mu\in H^2(-1/2,1/2):\mu'(-1/2)=\mu'(1/2)=0,\int_{-1/2}^{1/2}\mu(x)\,dx=0\right\}
			\end{equation}
			and
			\begin{equation}
				Y=\left\{m\in L^2(-1/2,1/2)\times\mathbb R\right\}.
			\end{equation}
			Define $\mathcal F:X\times\mathbb R\to Y$ by
			\begin{equation}\label{F}
				\mathcal F(\mu,\tau)=\mu''+\phi'(1/2)\mu'-((\mu+1)\phi')'.
			\end{equation}
			where $\phi$ is satisfies $-Z\phi''+\phi=(P_0+\tau)(1+\mu)$ with periodic boundary conditions in $(-1/2,1/2)$. Observe that 
			\begin{equation}
				m(x,t)=\mu(x-\phi'(1/2)t,t)+1
			\end{equation}
			is a traveling wave solution to \eqref{eq:1D_m} if and only if $F(\mu,\tau)=0$. If $\phi'(1/2)=0$, then $m$ is a stationary solution, i.e., a traveling wave with velocity $0$.
			
			We will show that $\mathcal F$ satisfies the properties  required for the validity of the CR theorem, see Appendix B.
			\begin{itemize}
				\item It is clear that $\mathcal F(0,\tau)=0$ for all $\tau$.
				\item It is also clear that $\mathcal F$ is twice continuously differentiable.
				\item The linearization of $\mathcal F$ in $\mu$ at $(\mu,\tau)=(0,0)$ is
				\begin{equation}\label{eq:bifurcation_linearization_null_space}
					0=D_{\mu}\mathcal F(0,0)u=u''-\psi''
				\end{equation}
				where $\psi$ satisfies $-Z\psi''+\psi=P_0u$ with periodic boundary conditions on $(-1/2,1/2)$. Note that this is the operator $S_C$ in \eqref{eq:SL_linearization}. To show that the third hypothesis of the CR theorem is satisfied, we need to show two things:
				\begin{enumerate}
					\item[(i)] There exists a unique (up to multiplicative constant) nonzero solution $u_0\in X$ to $D_\mu\mathcal F(0,0)u_0=0$, and
					\item[(ii)] there exists a co-dimension one subspace $X'$ of $X$ such that if $w\in X'$, then there exists a solution $u\in X$ to $D_\mu\mathcal F(0,0)u=w$.
				\end{enumerate}
				
				First we show (i). Consider the function
				\begin{equation}
					u_0(x)=\frac{\sqrt{Z} \text{sech}\left(\frac{\sqrt{1-P_0}}{2 \sqrt{Z}}\right) \sinh \left(\frac{\sqrt{1-P_0} x}{\sqrt{Z}}\right)}{(1-P_0)^{3/2}}+\frac{x}{P_0-1}.
				\end{equation}
				Observe that $u_0'(\pm 1/2)=0$ and $\int_{-1/2}^{1/2}u_0(x)\,dx=0$, so $u_0\in X$. Moreover, one may check that provided $P_0$ satisfies \eqref{eq:bifurcation_condition}, then $D_{\mu}\mathcal F(0,0)u_0=0$. Now suppose that $u_1$ and $u_2$ are both nonzero solutions to \eqref{eq:bifurcation_linearization_null_space}. Let $\psi_i$ solve $-Z\psi_i''+\psi_i=P_0u_i$ with periodic boundary conditions on $(-1/2,1/2)$ for $i=1,2$. Observe that for each $i$, since the second derivatives of $u_i$ and $\psi_i$ are equal, we have
				\begin{equation}
					u_i-\psi_i=\alpha_i x-\beta_i
				\end{equation}
				We also observe that 
				\begin{equation}
					\int_{-1/2}^{1/2}\psi_i\,dx=P_0\int_{-1/2}^{1/2}u_i\,dx+Z\int_{-1/2}^{1/2}\psi_i''\,dx=0.
				\end{equation}
				Therefore, $\beta_1=\beta_2=0$. Now, suppose for some $i$, $\psi_i'(\pm 1/2)=0$. Then $\alpha_i=0$ and $u_i=\psi_i$. Thus,
				$u_i$ satisfies $-Zu_i''+u_i=P_0u_i$ with periodic boundary conditions. Thus, $u_i$ is an eigenvector of the second derivative operator with eigenvalue $(1-P_0)/Z$. The eigenvalues of the second derivative operator on $X$ are $-n^2\pi^2$ for positive integers $n$. By assumption, $(1-P_0)/Z\neq -n^2\pi^2$, so this is a contradiction. Therefore, $\psi_i'(\pm 1/2)\neq 0$, and we may assume without loss of generality that the $u_i$ are scaled such that $\psi_i'(\pm 1/2)=1$ for $i=1,2$, so $\alpha_i=-1$. Then $u_i=\psi_i-x$. Let $w=u_1-u_2$. Then
				\begin{equation}
					w=(\psi_1-x)-(\psi_2-x)=\psi_1-\psi_2.
				\end{equation}
				Thus, $w$ satisfies $-Zw''+w=P_0w$ with periodic boundary conditions on $(-1/2,1/2)$. Once again, since $(1-P_0)/Z\neq -n^2\pi^2$, the only solution is $w=0$. We conclude that $u_1=u_2$ and $u_0$ is unique up to a multiplicative constant.
				
				Now we show (ii). Observe that since $\psi$ satisfies $\psi''=(P_0u-\psi)/Z$, we may write
				\begin{equation}
					D_m\mathcal F(0,0)u=u''-\frac{P_0}{Z}u+\frac{1}{Z}\psi.
				\end{equation}
				We may abstract this operator as $D_m\mathcal F(0,0)=B+K$ where $Bu=u''-(P_0/Z)u$ and $K:u\mapsto \psi$. We make two observations. First, $K$ is a bounded operator. This follows from standard elliptic estimates and is proved in Proposition \ref{prop:elliptic_estimates} in in Appendix A. Second, the operator $B$ is invertible on $X$ and its inverse $B^{-1}$ is compact. Therefore, we may write
				\begin{equation}
					I-B^{-1}D_\mu\mathcal F(0,0)=I-B^{-1}(B+K)=-B^{-1}K
				\end{equation}
				and
				\begin{equation}
					I-D_\mu\mathcal F(0,0)B^{-1}=I-(B+K)B^{-1}=-KB^{-1}.
				\end{equation}
				Since $B^{-1}$ is compact, so are $B^{-1}K$ and $KB^{-1}$. We conclude that $D_\mu\mathcal F(0,0)$ is a Fredholm operator. We recall that the index of a Fredholm operator is the difference between the dimension of its kernel and the codimension of its range. We also recall that the index of a self-adjoint operator is zero. Since $B$ and $K$ are both self-adjoint over $L^2([-1/2,1/2])$, so is $D_\mu\mathcal F(0,0)$. Therefore, the codimension of the range of $D_\mu \mathcal F(0,0)$ is equal to the dimension of the kernel, which we have just proved is $1$.
				
				\item Finally, we must show that $D_{\mu s}\mathcal F(0,0)u_0$ is not in the range of the operator $D_{\mu}\mathcal F(0,0)$. Observe that since $D_\mu \mathcal F(0,0)$ is self-adjoint, its image is orthogonal to its kernel. That is, for any $u\in X$, $\langle D_\mu\mathcal F(0,0)u,u_0\rangle_{L^2}=0$. It is therefore sufficient to show that $\langle D_{\mu s} \mathcal F(0,0)u_0,u_0\rangle_{L^2}\neq0$. The mixed second derivative is
				\begin{equation}
					D_{\mu s} \mathcal F(0,0)u=u''-\tilde\psi''
				\end{equation}
				where $\tilde\psi$ satisfies $-Z\tilde\psi''+\tilde\psi=u$. Therefore, $\tilde\psi=\psi/P_0$. We conclude that
				\begin{align*}
					\langle D_{\mu s}\mathcal F(0,0)u_0,u_0\rangle_{L^2}=\int_{-1/2}^{1/2}\left(u_0''-\frac{\psi_0''}{P_0}\right)u_0\,dx=\int_{-1/2}^{1/2}\left(1-\frac{1}{P_0}\right)u_0u_0''\,dx=\int_{-1/2}^{1/2}\left(1-\frac{1}{P_0}\right)(u_0')^2\,dx.
				\end{align*}
				Since $P_0\neq 1$, $\langle D_{\mu s}\mathcal F(0,0)u_0,u_0\rangle_{L^2}\neq 0$.
			\end{itemize}
			Since we now checked that all the hypotheses of the CR theorem hold  in a neighborhood of $(\mu,\tau)=(0,0)$, the only solutions to $\mathcal F(\mu,\tau)=0$ are $\mu=0$ plus a smooth family of solutions $(\mu(s),\tau(s))$ parameterized by $s$ in some small interval $(-s^\ast,s^\ast)$ with $\mu(s)\not\equiv 0$. Moreover, these two families of solutions meet at $(0,0)$. Indeed, let $m_T(s)=1+\mu(s)$ and $P_T(s)=P_0+\tau(s)$. Since all solutions to $\mathcal F=0$ are traveling waves with some velocity (or stationary solutions if the velocity is zero), it only remains to show that $m_T(s)$ and $P_T(s)$ may be reparameterized (at least locally near $s=0$) by velocity. Let $V(s)$ be the velocity of $m_T(s)$. It is sufficient to show that $V'(0)\neq 0$. Let $\phi_T(s)$ satisfy	$-Z\phi_T''+\phi_T=P_T(s)m_T(s)$ with periodic boundary conditions on $(-1/2,1/2)$. Then $V(s)=\partial_x\phi_T(s)|_{x=1/2}$. Therefore, $V'(0)$ is $\psi'(1/2)$ where $\psi$ solves
			\begin{equation}
				-Z\psi''+\psi=\frac{\partial}{\partial s}P_T(s)m_T(s)\Big|_{s=0}=\frac{\partial P_T}{\partial s}(0)m_T(0)+P_T(0)\frac{\partial m_T}{\partial s}(0)=\frac{\partial P_T}{\partial s}(0)+P_0\frac{\partial \mu}{\partial s}(0),
			\end{equation}
			with periodic boundary conditions. Since $\mathcal F$ is twice differentiable, $\mu(s)$ is also continuously differentiable and $\mu'(0)$ spans the null space of $\mathcal F_\mu(0,0)$. Without loss of generality, we may assume that $\mu(s)$ is parameterized such that
			\begin{equation}
				\frac{d\mu}{ds}(0)=\mu_0.
			\end{equation}
			Let $\psi_0$ solve $-Z\psi_0''+\psi_0=P_0\mu_0$ with periodic boundary conditions on $(-1/2,1/2)$. Then $\psi_0$ and $\psi$ differ by a constant, so $V'(0)=\psi_0'(1/2)$. We may explicitly calculate $\psi_0'(1/2)=1$, so $V'(0)\neq 0$. Therefore, we may reparameterize $m_T$ and $P_T$ by $V$ for $V$ in some small interval $(-V^\ast,V^\ast)$.
		\end{proof}
		
		\begin{remark}\label{rem:smooth_TWs}
			All the analysis in the proof of Theorem \ref{thm:bifurcation} can be done if $X$ is replaced with a more restrictive domain, such as a set of $C^r$ functions. Therefore, we conclude that the traveling wave solutions $m_T$, $\phi_T$ are infinitely differentiable and depend smoothly on $V$ in the $C^r$-norm for any $r$.
		\end{remark}
		Given $Z>0$, the condition \eqref{eq:bifurcation_condition} satisfied by $P_0$ has (potentially) infinitely many solutions. Therefore, Theorem \ref{thm:bifurcation} proves the existence of not just one, but infinitely many families of traveling wave solutions, each bifurcating from the stationary solution for a different solution $P$ to \eqref{eq:bifurcation_condition}. In  Section \ref{sec:stationary_solution_stability}, we observed that for $P/Z<\pi^2$, the eigenvalues of the linearization $S_C$ of model $C$ about the stationary solution are all negative. In the proof of Theorem \ref{thm:bifurcation}, we observe that if $P$ satisfies \eqref{eq:bifurcation_condition}, $S_C$ has a zero eigenvalue. We conclude that as $P/Z$ increases from $\pi^2$, each solution of \eqref{eq:bifurcation_condition} corresponds to one of the eigenvalues of $S_C$ becoming positive. Therefore, we conjecture that for all families of traveling waves except those bifurcating from the smallest solution of \eqref{eq:bifurcation_condition}, the linearization of model C about these traveling waves has some positive eigenvalues, and therefore these traveling waves are unstable. The only traveling wave solutions that may be stable are those bifurcating from the smallest solution to \eqref{eq:bifurcation_condition}. Therefore, when using the notation $P_0$, we refer to this value. The following Lemma shows the existence of this smallest solution and provides the illuminating estimate that, for large $Z$, $P_0/Z\geq\pi^2$ with equality in the limit $Z\to\infty$.
		\begin{lemma}\label{lem:P0_asymptotic_form}
			Suppose that $P_0=P_0(Z)$ is the smallest positive solution to
			\begin{equation}\label{eq:bifurcation_condition_2}
				\tanh\left(\frac{\sqrt{1-P_0}}{2\sqrt{Z}}\right)=\frac{P_0\sqrt{1-P_0}}{2\sqrt{Z}}
			\end{equation}
			besides $P_0=1$ whenever such a solution exists. Then $P_0(Z)$ exists for all $Z>0$ except $Z=1/12$ and in large $Z$, $P_0(Z)$ expands as
			\begin{equation}
				P_0(Z)=\pi^2Z+1-\frac{8}{\pi^2}+O(1/Z).
			\end{equation}
		\end{lemma}
		\begin{proof}
			First suppose $0<Z<1/12$. Let $v=\sqrt{1-P_0}/(2\sqrt{Z})$. Then $v$ and $Z$ satisfy
			\begin{equation}\label{eq:k1}
				k_1(z):=\frac{v-\tanh(v)}{4v^3}=Z.
			\end{equation}
			It is easy to show that $k_1$ is continuous on $(0,\infty)$, $\lim_{v\to 0^+}k_1(v)=1/12$, $\lim_{v\to\infty}k_1(v)=0$, and $k_1$ is monotonically decreasing. Thus, for any $Z\in (0,1/12)$, there exists a unique $v\in (0,\infty)$ satisfying \eqref{eq:k1}. Therefore, $P_0(Z)=\tanh(v)/v\in(0,1)$ is uniquely determined.
			
			Now suppose $Z>1/12$. Write \eqref{eq:bifurcation_condition_2} as
			\begin{equation}
				\tan\left(\frac{\sqrt{P_0-1}}{2\sqrt{Z}}\right)=\frac{P_0\sqrt{P_0-1}}{2\sqrt{Z}}
			\end{equation}
			and let $w=\sqrt{P_0-1}/(2\sqrt{Z})$. Then $w$ and $Z$ satisfy
			\begin{equation}\label{eq:k2}
				k_2(w)=\frac{\tan(w)-w}{4w^3}=Z.
			\end{equation}
			Similarly to $k_1$, it is easy to show that $k_2$ is continuous on $(0,\pi/2)$, $\lim_{w\to 0^+}k_2(w)=1/12$, $\lim_{w\to\pi/2^-}k_2(w)=\infty$, and $k_2$ is monotonically increasing on $(0,\pi/2)$. Thus, for any $Z\in (1/12,\infty)$, there exists a unique $w\in(0,\pi/2)$ satisfying \eqref{eq:k2}. Therefore, $P_0(Z)=\tan(w)/w\in (1,\infty)$. Thus, for all positive $Z$ other than $Z=1/12$, \eqref{eq:bifurcation_condition_2} has a smallest positive solution other than $1$. It should be noted that using a similar line of reasoning, we may show that \eqref{eq:k2} has a unique solution in each interval of the form $(n\pi/2,(n+2)\pi/2)$ for $n>0$ odd. These correspond to the other (larger than $P_0$) solutions to \eqref{eq:bifurcation_condition} referenced above.
			
			As $Z\to\infty$ the corresponding solution $w$ to $k_2(w)=Z$ approaches $\pi/2$. Therefore, we expand $w$ in large $Z$ as $w(Z)=\pi/2+w_1/Z+w_2/Z^2+O(1/Z^3)$. We expand \eqref{eq:k2} in large $Z$ and compare terms of like order in $Z$ to obtain $w=\pi/2-2/(Z\pi^3)+O(1/Z^2)$. Finally, using $P_0=1+4Zw^2$, we have
			\begin{equation}
				P_0=\pi^2Z+\left(1-\frac{8}{\pi^2}\right)+O(1/Z).
			\end{equation}
		\end{proof}

	\section{Stability of Traveling Waves}
	
	In this section we study the nonlinear stability of traveling wave solutions to Model B. As shown in Theorem \ref{thm:bifurcation}, traveling wave solutions of velocity $V$ sufficiently small exist provided $P$ has the prescribed value $P_T(V)$. Such a traveling wave solution has the form $m(x,t)=m_T(x-Vt)$ with $c=c_0+Vt$ where $m_T$ is a stationary solution to model $C$. For ease of analysis we will study the stability of these solutions in the framework of model C which, as described in Section \ref{sec:stationary_solution_stability}, implies  stability ``up to shifts" of traveling wave solutions to model B.
	
	As in Section \ref{sec:stationary_solution_stability}, we describe model C by the dynamical system $\partial_t m=F_C(m)$ with $F_C$ given by \eqref{eq:model_C}. The linearization about $m_T$ about traveling waves is:
	\begin{equation}\label{eq:SL_TW_linearization}
		T_Cu:=DF_C(m_T)=u''+\phi'(1/2)m_T'+V	u'-(m_T\phi')'-(u\phi_T')',
	\end{equation}
	where $\phi$ and $\phi_T$ satisfy respectively
	\begin{equation}\label{eq:linearized_phi_and_phi_T}
		\begin{cases}
			-Z\phi''+\phi=Pu & -1/2<x<1/2\\\phi(1/2)=\phi(-1/2)\\\phi_x(-1/2)=\phi_x(1/2)
		\end{cases}\quad\text{and}\quad\begin{cases}
		-Z\phi_T''+\phi_T=Pm_T & -1/2<x<1/2\\ \phi_T(1/2)=\phi_T(-1/2)\\\phi_{T,x}(-1/2)=\phi_{T,x}(1/2),
		\end{cases}
	\end{equation}
	and $u\in\tilde X^2_C:=\left\{m\in H^2(-1/2,1/2):m_x(\pm 1/2)=0,\,\int_{-1/2}^{1/2}u\,dx=0\right\}$. The coefficient $V$ appears because the velocity of the traveling wave is $V=\phi_T'(1/2)$.
	
	Also of interest is the nonlinear part of $F_C$ near $m=m_T$. However, due to the quadratic nature of the nonlinearities in $F_C$, this nonlinear part has the exact same form as the nonlinear part of $F_C$ about $m=1$:
	\begin{equation}
		\Psi(u)=F_C(m_T+u)-T_C u=\phi'(1/2)u'-(u\phi')',
	\end{equation}
	with $\phi$ given in \eqref{eq:linearized_phi_and_phi_T}. Using the nonlinear part $\Psi$, and letting $u=m-m_T$, we may rewrite the evolution equation \eqref{eq:model_C} as
	\begin{equation}
		\partial_t u=T_Cu+\Psi(u).
	\end{equation}
	
	
	Similar to Section \ref{sec:stationary_solution_stability}, our analysis in this section is focused on proving two key results:
	\begin{enumerate}
		\item $0$ is an asymptotically stable equilibrium of the linearized problem $u_t=T_Cu$, and
		\item $m_T$ is an asymptotically stable equilibrium of the nonlinear problem $m_t=F_C(m)$.
	\end{enumerate}
	In this section however, a new challenge arises: the operator $T_C$ is non self-adjoint, meaning that the spectral theorem used in the proof of linear stability in Theorem \ref{thm:SS_linear_stability} no longer applies.

	Indeed, while, as we have already mentioned in the Introduction,  a self-adjoint operator with compact inverse has a basis of eigenvectors, no such basis is guaranteed if the operator is non self-adjoint operator, meaning that there may be a portion of the domain of the operator   hidden from the eigenvectors. Since the action of the operator on this ``dark'' space cannot be determined from the eigenvectors, it is not sufficient merely to show that all the eigenvalues of the operator have negative real part. Instead, we rely on the GPG theorem \cite{EngNagBre2000}, which we quote for convenience in Appendix C. 
	
	The GPG theorem  overcomes the problem  with invisibility of a part of the domain by considering not just eigenvalues, but the entire  spectrum  of the operator. The spectrum of a linear operator $L$ is the set of all $\lambda\in\mathbb C$ so that the operator $\lambda I-L$ does not have a bounded inverse. Note that if $\lambda I-L$ is not invertible because it is not injective (one-to-one), then $\lambda$ is an eigenvalue of $L$. 
	
	We recall that if $L$ is a finite dimensional linear operator (a matrix), then the rank-nullity theorem applies and $\lambda I-L$ is invertible if and only if it is injective. In the infinite dimensional case, however, a linear operator may be injective but not surjective, and thus not invertible. Even if $\lambda I-L$ is invertible, its inverse may not bounded. Thus, the spectrum of $L$ may consist of more than just eigenvalues. 
	
	We also recall that if $\lambda I-L$ does have a bounded inverse, $(\lambda I-L)^{-1}$ is called a  resolvent  operator of $L$, and the set of $\lambda$ such that the resolvent exists (that is the complement of the spectrum) is called the  resolvent set. The solution $x(t)$ to the linear system $x_t=Lx$ can be written in terms of the resolvent via a line integral in the complex plane as an inverse Fourier Transform:
	\begin{equation}\label{eq:resolvent_formula}
		x(t)=\lim_{s\to\infty}\frac{1}{2\pi i}\int_{w-is}^{w+is}e^{\lambda t}(\lambda I-L)^{-1}x(0)\,d\lambda
	\end{equation}
	for $w\in\mathbb R$ sufficiently large \cite{EngNagBre2000}. Furthermore, if $x(0)$ can be written $x(0)=\sum_{n=1}^\infty c_n x_n$ for eigenvectors $x_n$ of $L$ with eigenvalues $\lambda_n$, then if $w>\sup_n\text{Re}\,\lambda_n$,
	\begin{align}
		x(t)&=\lim_{s\to\infty}\frac{1}{2\pi i}\int_{w-is}^{w+is}e^{\lambda t}(\lambda I-L)^{-1}\sum_{n=1}^\infty c_nx_n\,d\lambda\\
		&=\sum_{n=1}^\infty c_nx_n\lim_{s\to\infty}\frac{1}{2\pi i}\int_{w-is}^{w+is}\frac{e^{\lambda t}}{\lambda-\lambda_n}\,d\lambda\\
		&=\sum_{n=1}^\infty c_nx_n\frac{e^{wt}}{2\pi }\int_{-\infty}^\infty\frac{e^{ist}}{w+is-\lambda_n}\,ds\\
		&=\sum_{n=1}^\infty c_nx_n\frac{e^{wt}}{2\pi}\left(2\pi e^{(\lambda_n-w)t}\right)\\
		&=\sum_{n=1}^\infty c_ne^{\lambda_nt}x_n.
	\end{align}
The crucial observation is that  \eqref{eq:resolvent_formula} holds even if $x(0)$ cannot be written as a sum of eigenvectors (i.e., if the eigenvectors of $L$ do not span the domain of $L$) provided the resolvent $((w+is)I-L)^{-1}$ exists for all $s\in\mathbb R$ and $w$ is sufficiently large. 

The GRG theorem (formulated in Appendix C) provides conditions on the resolvent and spectrum of $L$ such that, via \eqref{eq:resolvent_formula}, all solutions $x(t)$ converge to $0$ exponentially fast. Then, since the solutions to our  linearized problem $u_t=T_Cu$ are $u(t)=S(t)u(0)$,  we can conclude $\lim_{t\to\infty}\Vert u(t)\Vert\leq\lim_{t\to\infty}e^{-\sigma t}\Vert u(0)\Vert=0$ whenever  the conditions of the GRG theorem hold. Then   $0$ is asymptotically stable in the linearized system. 

In view of the above, to establish linear stability, it remains to be shown that each of the three conditions of the GPG theorem  hold for the operator $T_C$. Those are identified  in our   Appendix C as conditions (i),(ii),(iii) and are also spelled out explicitly below. We will begin with proving condition (ii), then condition (i), and finally condition (iii). Before proceeding with these steps,   we first show that $T_C$ is non self-adjoint.
		
	\begin{theorem}
		There exists $Z^\ast$, $V^\ast>0$ such that if $Z>Z^\ast$ and $0<|V|<V^\ast$, the operator $T_C$ is non self-adjoint.
	\end{theorem}
	\begin{proof}
		As in the proof of Lemma \ref{lem:SS_self-adjoint}, we will use the adjoint commutator. We will show that there exists $V^\ast>0$ and $u_1,u_2\in \tilde X^2_C$ so that if $0<|V|<V^\ast$, then the adjoint commutator $H$ for the operator $T_C=T_C(V)$ evaluated at $u_1,u_2$ is nonzero. This shows that $T_C(V)$ is non self-adjoint.
		
		Let $u_1(x)=\sin(\pi x)$ and $u_2(x)=\cos(2\pi x)$. Both $u_1$ and $u_2$ are in $\tilde X^2_C$. For $i=1,2$ et $\psi_i$ satisfy $-Z\psi_i''+\psi_i=P_T(V)u_i$ with periodic boundary conditions in $(-1/2,1/2)$. Then
		\begin{equation}
			\begin{split}
				H(u_1,u_2)&=\int_{-1/2}^{1/2}u_1u_2''-u_2u_1''\,dx+\int_{-1/2}^{1/2}m_T'(u_1\psi_2'(1/2)-u_2\psi_1'(1/2))\,dx+V\int_{-1/2}^{1/2}u_1u_2'-u_2u_1'\,dx\\
				&-\int_{-1/2}^{1/2}u_1(m_T\psi_2')'-u_2(m_T\psi_1')'\,dx-\int_{-1/2}^{1/2}u_1(u_2\phi_T')'-u_2(u_1\phi_T')'\,dx.
			\end{split}
		\end{equation}
		The function $\psi_1,\psi_2$ depend on $V$ through $P_T(V)$, but from Lemma \ref{lem:TW_asymptotic_form}, $P_T'(0)=0$, so $\partial_V\psi_i|_{V=0}=0$. Also from Proposition \ref{lem:TW_asymptotic_form}, the traveling wave solution satisfy $\partial_V m_T|_{V=0}=m_1$ given by \eqref{eq:m1} and $\partial_V\phi_T|_{V=0}=m_1(x)+x$. Therefore,
		\begin{equation}\label{eq:H_derivative}
			\begin{split}
				\partial_VH(u_1,u_2)|_{V=0}&=\int_{-1/2}^{1/2}u_1u_2''-u_2u_1''\,dx+\int_{-1/2}^{1/2}m_1'(u_1\psi_2'(1/2)-u_2\psi_1'(1/2))\,dx+\int_{-1/2}^{1/2}u_1u_2'-u_2u_1'\,dx\\
				&-\int_{-1/2}^{1/2}u_1(m_1\psi_2')'-u_2(m_1\psi_1')'\,dx-\int_{-1/2}^{1/2}u_1(u_2\phi_1')'-u_2(u_1\phi_1')'\,dx.
			\end{split}
		\end{equation}
		Since each of the functions $m_1$, $\phi_1$, $u_1$, $u_2$, $\psi_1$, and $\psi_2$ are explicitly known, and using $P_0=\pi^2Z+O(1)$ from Lemma \ref{lem:P0_asymptotic_form}, we may explicitly calculate the integrals in \eqref{eq:H_derivative} and find the asymptotic expansion of the result in large $Z$:
		\begin{equation}
			\partial_VH(u_1,u_2)|_{V=0}=-3+O(1/Z).
		\end{equation}
		Therefore, for sufficiently large $Z^\ast$, if $Z>Z^\ast$, then $\partial_VH(u_1,u_2)|_{V=0}\neq 0$. Thus, there exists $V^\ast$ so that if $0<|V|<V^\ast$, $H(u_1,u_2)\neq 0$. We conclude that for $Z>Z^\ast$ and $0<|V|<V^\ast$, $A_T(V)$ is non self-adjoint. 
	\end{proof}
	
	\paragraph{Condition (ii): the resolvent set of $T_C$ contains the right half-plane, see appendix C. }	
	To establish that condition (ii) holds, we prove a sequence of four results. First, we show that the resolvent of $T_C$, if it exists, is compact. Then we show that for some $\lambda_0>0$, there exists a unique weak solution to $(\lambda_0-T_C)u=w$ for each $w$, which implies that the resolvent $(\lambda_0I-T_C)^{-1}$ exists. Next, we use the first two results to show that the spectrum of $T_C$ consists only of its eigenvalues. Finally, we show that all the eigenvalues of $T_C$ have negative real part. Thus, the resolvent set contains all complex numbers with positive real part, and condition (ii) is satisfied.
	
	\begin{proposition}\label{prop:TW_compact_inverse}
		Suppose $\lambda\in C$ such that $\lambda I-T_C$ is invertible. Then $(\lambda I-T_C)^{-1}:L^2(-1/2,1/2)\to L^2(-1/2,1/2)$ is a compact operator. 
	\end{proposition}
	\begin{proof}
		To be compact, $(\lambda I-T_C)^{-1}$ must be bounded. Suppose, to the contrary that it is unbounded. Then there exist sequences $(v_k)\subset \tilde X^2_C$ and $(w_k)\subset L^2(-1/2,1/2)$ such that
		\begin{equation}
			(\lambda I-T_C)v_k=w_k,\quad \Vert v_k\Vert_{L^2}=1,\quad\Vert w_k\Vert_{L^2}\leq 1/k.
		\end{equation}
		Let $\phi_k$ satisfy $-Z\phi_k''+\phi_k=Pv_k$ with periodic boundary conditions. Then the following sequence is bounded:
		\begin{equation}\label{eq:TW_compact_resolvent_inner_product}
			\begin{split}
				\langle w_k,v_k\rangle_{L^2}&=\langle (\lambda I-T_C)v_k,v_k\rangle_{L^2}\\
				&=\lambda\Vert v_k\Vert_{L^2}^2+\Vert v_k'\Vert_{L^2}^2+\langle (V-\phi_T') v_k',v_k\rangle_{L^2}+\langle (\phi_k'(1/2)-\phi_k')m_T',v_k\rangle_{L^2}\\
				&\hspace{2in}-\langle\phi_T''v_k,v_k\rangle_{L^2}-\langle \phi_k''m_T,v_k\rangle_{L^2}.
			\end{split}
		\end{equation}
		Every term in this sequence is individually bounded due to Proposition \ref{prop:elliptic_estimates} in Appendix A except possibly $\Vert v_k'\Vert_{L^2}^2$ and $\langle (V-\phi_T') v_k',v_k\rangle_{L^2}$. However, the sum of these terms must be bounded. While the former is quadratic in $\Vert v_k'\Vert_{L_2}$, the latter is at most linear. Therefore, they must both be independently bounded as well.
		
		Since $\Vert v_k\Vert_{L^2}$ and $\Vert v_k'\Vert_{L^2}$ are both bounded, we conclude that $(v_k)$ is bounded with respect to the $H^1$ norm. The remaining arguments giving rise to a contradiction and proving that $(\lambda I-T_C)^{-1}$ is bounded and, moreover, compact, are identical to those in the proof of Lemma \ref{lem:SS_inverse_compact}.
	\end{proof}
	
	\begin{proposition}\label{prop:exists_inverse}
		There exists $V^\ast>0$ and $\lambda_0\geq 0$ such that for each $w\in X^0$, there exists a unique weak solution $u\in X^1$ to $T_Cu=w$. 
	\end{proposition}
	\begin{proof}
		Define the bilinear form $B:X^1\times X^1\to\mathbb R$ by
		\begin{equation}
			B[u,v]=\langle u',v'\rangle_{L^2}-\langle (\phi'(1/2)-\phi')m_T'+(V-\phi_T')u'-m_T\phi''-u\phi_T'',v\rangle_{L^2}.
		\end{equation}
		Then $u\in X^1$ is a weak solution to $T_Cu=w$ if and only if $B[u,v]=\langle w,v\rangle_{L^2}$ for all $v\in X^1$. We claim that there exist $a,b,V^\ast>0$ and $\lambda_0\geq 0$ such that
		\begin{itemize}
			\item $|B[u,v]|\leq a\Vert u\Vert_{H^1}\Vert v\Vert_{H^1}$
			\item $b\Vert v\Vert^2_{H^1}\leq B[v,v]+\lambda_0\Vert v\Vert_{L^2}^2$.
		\end{itemize}
		The proof of these facts follows from the Poincar\'e inequality and the fact that $\Vert \phi_T'\Vert_{L^2}=O(V)$. Therefore, by the Lax-Milgram Theorem, there exists a unique weak solution to $(\lambda_0 I-T_C)u=w$.
	\end{proof}
	
	\begin{proposition}\label{prop:spectrum_is_eigenvalues}
		The spectrum of $T_C$ consists only of its eigenvalues.
	\end{proposition}
	\begin{proof}
		This proof is essentially showing that the Fredh\"olm alternative applies to $T_C$. Let $\lambda\in\mathbb C$, and let $\lambda_0$ be defined as in \ref{prop:exists_inverse}. Define $\hat T_C=\lambda_0 I-T_C$ and let $\lambda'=\lambda_0-\lambda$. Then $\lambda I-T_C=\hat T_C-\lambda'I$. By Proposition \ref{prop:exists_inverse}, $\hat T_C$ is invertible, and by Proposition \ref{prop:TW_compact_inverse}, $\hat T_C^{-1}$ is compact. Therefore, we may apply the Fredh\"olm alternative for compact operators \cite{Dou2012} to see that exactly one of the following holds:
		\begin{itemize}
			\item $(I-\lambda'\hat T_C^{-1})v=\hat T_C^{-1}w$ has a unique solution for each $w\in X^0$,
			\item $(I-\lambda'\hat T_C^{-1})v=0$ has a nontrivial solution.
		\end{itemize}
		In either case, we may multiply by $\hat T_C$ to see that either $(\lambda I-T_C)v=w$ has a unique solution for all $w\in X^0$ or $(\lambda I-T_C)v=0$. Therefore, either $\lambda I-T_C$ is invertible (with bounded inverse per Proposition \ref{prop:TW_compact_inverse}) and therefore $\lambda$ is not in the spectrum, or $\lambda$ is an eigenvalue of $T_C$. Therefore, the spectrum of $T_C$ consists only of its eigenvalues.
	\end{proof}
	
	The following lemma show that all eigenvalues of $T_C$ have negative real part except possibly one. The following Theorem concerns this remaining eigenvalue showing that it too has negative real part, thus proving the desired result.
	
	\begin{lemma}\label{lem:mostly_negative_evs}
		For $V$ sufficiently small the eigenvalues of $T_C=T_C(V)$ all have negative real part bounded away from $0$ except possibly one. Moreover, when $V=0$, all the eigenvalues of $T_V(0)$ are negative (and real) except for a zero eigenvalue with multiplicity 1.
	\end{lemma}
	\begin{proof}
		The domain of $T_C(V)$ is $\tilde X^2_C$, which has the (Schauder) basis $\mathcal B=\{v_1,v_2,v_3,\cdots\}$ where $v_n(x)=\sin(n\pi x)$ for $n$ odd, and $v_n(x)=\cos(n\pi x)$ for $n$ even. For each $n,m\in\mathbb N$, define
		\begin{equation}
			a_{mn}=\langle v_m, T_C(V)v_n\rangle_{L^2}
		\end{equation}
		Treating $A=(a_{mn})$ as an ``infinite matrix'' operator on $\ell^2$, we see that $\lambda$ is an eigenvalue of $T_C(V)$ if and only if $\lambda/2$ it is an eigenvalue of $A$. In particular, the eigenvalues of $A$ and $T_C(V)$ have the same sign.
		
		Many of the terms in $T_C(V)$ vanish as $V\to 0$. In particular, the traveling waves $m_T$ and $\phi_T$ and their derivatives depend smoothly on $V$ in $L^2$. Moreover, when $V=0$, $m_T$ and $\phi_T$ are both constant in $x$ (they are stationary states). Therefore, writing $m_T=1+\tilde m_T$, there exists $C_1,V^\ast>0$ so that if $|V|<V^\ast$, then  
		\begin{equation}\label{eq:TW_controlled_by_V}
			\Vert m_T'\Vert_{L^1},\Vert\phi_T'\Vert_{L^1},\Vert\phi_T''\Vert_{L^1},\Vert\tilde m_T\Vert_{L^1}\leq C_1|V|.
		\end{equation}
		For each $m$, let $\phi_m$ solve $-Z\phi_m''+\phi_m=P_T(V)v_m$ with periodic boundary conditions in $(-1/2,1/2)$. For each $n,m$, define 
		\begin{equation}
			d_{mn}=\langle v_m,\phi_n'(1/2)m_T'+Vv_n'-\tilde m_T\phi_n''-m_T'\phi_n'-v_n'\phi_T'-v_n\phi_T''\rangle_{L^2}.
		\end{equation}
		From \eqref{eq:TW_controlled_by_V} and Lemma \ref{lem:inner_pruduct_perturbation}, there exists $C>0$ independent of $m$ such that
		\begin{equation}
			\sum_{n=1}^\infty|d_{mn}|\leq C|V|\quad\text{and}\quad\sum_{m=1}^\infty|d_{mn}|\leq C|V|.
		\end{equation}
		Then we may write for each $n,m$:
		\begin{equation}
			a_{mn}=\langle v_m,T_C(0)v_n\rangle_{L^2}+d_{mn}.
		\end{equation}
		The operator $T_C(0)$ (which is equal to $S_C(P_0)$) is defined by $T_C(0)u=u''-\phi''$ where $\phi$ solves $-Z\phi''+\phi=P_0u$ with periodic boundary conditions in $(-1/2,1/2)$. Thus, letting $c_{mn}=\langle v_n,T_C(0)v_m\rangle_{L^2}$, we have $a_{mn}=c_{mn}+d_{mn}$. We can explicitly calculate $c_{mn}$:
		\begin{equation}
			c_{mn}=\begin{cases}
				-n^2\pi^2+\frac{P_0}{Z}\frac{1}{1+\frac{1}{\pi ^2 n^2 Z}} & n=m\text{ even}\\
				-n^2\pi^2+\frac{P_0}{Z}\frac{1}{1+\frac{1}{\pi ^2 n^2 Z}}-\frac{4P_0\coth\left(\frac{1}{2\sqrt{Z}}\right)}{\sqrt{Z}(1+n^2\pi^2Z)^2} & n=m \text{ odd}\\
				0 & n\neq m \text{ either $m$ or $n$ even}\\
				-\frac{4 P_0 (-1)^{\frac{m+1}{2}+\frac{n+1}{2}} \coth \left(\frac{1}{2 \sqrt{Z}}\right)}{\sqrt{Z} \left(\pi ^2 m^2 Z+1\right) \left(\pi ^2 n^2 Z+1\right)} & n\neq m \text{ both odd}.
			\end{cases}
		\end{equation}
		
		To show that all the eigenvalues of $A$ are negative except possibly one of them we will use Theorem 3 of \cite{ShiWilRud1987}, which gives a Gershgorin-type result  showing that all eigenvalues of an infinite matrix have negative real part. While possibly not all eignevalues of $A$ have negative real part, we will see using Theorem 3 of \cite{ShiWilRud1987} that all eigenvalues of $D=B-I$ do have negative real part, and that all but one of these eigenvalues has real part less than $-1$, thus proving the desired result.
		
		The specific result of Theorem 3 of \cite{ShiWilRud1987} is that there are countably many eigenvalues $\hat\lambda_n$ of $B=(b_{mn})$ and for each $n$,
		\begin{equation}
			|\hat\lambda_n-b_{nn}|<Q_n:=\sum_{\substack{m=1\\ m\neq n}}^\infty|b_{mn}|,
		\end{equation}
		provided the following conditions are met:
		\begin{enumerate}
			\item $b_{nn}\neq 0$ for any $n$ and $\lim_{n\to\infty}|b_{nn}|=\infty$.
			\item There exists $0<\rho<1$ so that for each odd $n$,
			\begin{equation}
				\frac{Q_n}{|b_{nn}|}<\rho.
			\end{equation}
			\item For each odd $n,m$ with $n\neq m$, $|b_{nn}-b_{mm}|\geq Q_n+Q_m$
			\item For each $m$, $\sup\{|b_{mn}|:n\in\mathbb N\}<\infty$.
		\end{enumerate}
		We will show that $B$ satisfies each of these conditions for small enough $V<V^\ast$ and large enough $Z$.
		\begin{enumerate}
			\item Observe that
			\begin{equation}
				b_{nn}<-n^2\pi^2+\frac{P_0}{Z}\frac{1}{\frac{1}{n^2\pi^2Z}+1}+C|V|-1.
			\end{equation}
			Let $0<\varepsilon<1/(2+2\pi^2)$. Using Lemma \ref{lem:P0_asymptotic_form}, there exists $Z^\ast$ large enough that for all $Z>Z^\ast$, $P_0/Z<\pi^2+\varepsilon/2$.  There also exists $V^\ast>0$ so that if $|V|<V^\ast$, $C|V|<\varepsilon/2$. Therefore, for large enough $Z$, $b_{nn}<-\pi^2(n^2-1)-1+\varepsilon<0$. It is clear that $\lim_{n\to\infty}|b_{nn}|=\infty$.
			
			\item We have
			\begin{equation}
				Q_n=\sum_{\substack{m=1\\ m\neq n}}^\infty|b_{mn}|\leq \sum_{\substack{m=1\\ m\neq n}}^\infty|c_{mn}|+\sum_{\substack{m=1\\ m\neq n}}^\infty|d_{mn}|
			\end{equation}
			If $n$ is even, $Q_n\leq\sum_{n=1}^\infty|d_{mn}|<C|V|<\varepsilon/2$. If $n$ is odd, we can explicitly calculate a convenient upper bound for $Q_n$:
			\begin{align}
				Q_n&<Q_n+\frac{4 P_0 \coth \left(\frac{1}{2 \sqrt{Z}}\right)}{\sqrt{Z} \left(\pi ^2 n^2 Z+1\right)^2}\\
				&\leq\sum_{\substack{m=1\\ m\text{ odd}}}^\infty\frac{4 P_0  \coth \left(\frac{1}{2 \sqrt{Z}}\right)}{\sqrt{Z} \left(\pi ^2 m^2 Z+1\right) \left(\pi ^2 n^2 Z+1\right)}+\sum_{m=1}^\infty|d_{mn}|\\
				&\leq \frac{P_0}{Z}\frac{1}{1+n^2\pi^2Z}+C|V|\\
				&<\frac{\pi^2+\varepsilon/2}{1+\pi^2}+\frac{\varepsilon}{2}\quad\text{assuming $Z^\ast\geq 1$}.\label{eq:Qn_bound}
			\end{align}
			We conclude that whether $n$ is even or odd, \eqref{eq:Qn_bound} is an upper bound for $Q_n$. We have already seen that for $Z>Z^\ast$, $b_{nn}<-1+\varepsilon$. Thus, for any $n$,
			\begin{equation}
				\frac{Q_n}{|b_{nn}|}<\frac{\pi^2+\varepsilon/2}{(1+\pi^2)(1-\varepsilon)}+\frac{\varepsilon/2}{1-\varepsilon}<\frac{\pi^2+\frac{1}{4+4\pi^2}}{(1+\pi^2)\left(1-\frac{1}{2+2\pi^2}\right)}+\frac{1}{(4+4\pi^2)\left(1-\frac{1}{2+2\pi^2}\right)}=\frac{2+5 \pi ^2+4 \pi ^4}{2+6 \pi ^2+4 \pi ^4}<1.
			\end{equation}
			Therefore, letting $\rho=\frac{2+5 \pi ^2+4 \pi ^4}{2+6 \pi ^2+4 \pi ^4}$, the second condition is satisfied.
			
			\item Let $n,m\in\mathbb N$ be odd with $n\neq m$. One can verify that for any $Z>0$,
			\begin{equation}
				0<\frac{4\sqrt{Z}\coth\left(\frac{1}{2\sqrt{Z}}\right)}{(1+\pi^2Z)^2}<1.
			\end{equation}
			Then if $Z>Z^\ast$ and $Z>Z^\ast$,
			\begin{align}
				|b_{nn}-b_{mm}|&\geq \pi^2|n^2-m^2|-\frac{P_0}{Z}\left|\frac{1}{\frac{1}{n^2\pi^2Z}+1}-\frac{1}{\frac{1}{m^2\pi^2Z}+1}\right|\\
				&\quad-\left|\frac{4P_0\coth\left(\frac{1}{2\sqrt{Z}}\right)}{\sqrt{Z}(1+n^2\pi^2Z)^2}-\frac{4P_0\coth\left(\frac{1}{2\sqrt{Z}}\right)}{\sqrt{Z}(1+m^2\pi^2Z)^2}\right|-2C|V|\\ 
				&\geq\pi^2|n^2-m^2|-2\frac{P_0}{Z}-\varepsilon\\
				&>3\pi^2-2(\pi^2-\varepsilon/2)-\varepsilon\\
				&>\pi^2-2\varepsilon.
			\end{align}
			On the other hand, we have seen that for each $n$, if $Z>Z^\ast$ and $|V|<V^\ast$, then $Q_n<1$, so $Q_n+Q_m<2<\pi^2-2\varepsilon$. Thus condition 3 is satisfied.
			
			\item This is clear.
		\end{enumerate}
		
		Thus, the eigenvalues of $B$ are enumerated $\hat\lambda_1,\hat\lambda_2,\hat\lambda_3,\cdots$, and for each $n$, $|b_{nn}-\hat\lambda_n|<Q_n$. Thus, for $n\geq 2$,
		\begin{equation}
			\text{Re}\,\hat\lambda_n<b_{nn}+Q_n<-4\pi^2+\frac{P_0}{V_0}+C|V|<-3\pi^2+\varepsilon<-1.
		\end{equation}
		Since the eigenvalues of $A$ are $\lambda_n=\hat \lambda_n+1$, we conclude that all $\lambda_n$ have negative real part bounded away from $0$ except for possibly $\lambda_1$. The eigenvalues of $T_C(V)$ are $2\lambda_n$ for $n=1,2,3,\cdots$, so the desired result holds.
		
		In the case $V=0$, the operator $T_C(0)$ is exactly the operator shown to have exactly one zero eigenvalue in the proof of Theorem \ref{thm:bifurcation}. Therefore, $T_C(0)$ has all negative eigenvalues (real because the operator is self-adjoint) except for one zero eigenvalue.
	\end{proof}
	
	\begin{theorem}\label{thm:negative_eigenvalues}
		There exists $V^\ast,Z^\ast>0$ such that if $0<|V|<V^\ast$ and $Z>Z^\ast$, then resolvent set of $T_C$ contains $\{z\in\mathbb C:\text{Re}\,z\geq0\}$.
	\end{theorem}
	\begin{proof}
		Due to Proposition \ref{prop:spectrum_is_eigenvalues}, we need only show that all eigenvalues of $T_C$ have negative real part. Lemma \ref{lem:mostly_negative_evs} gives $V^\ast$ and $Z^\ast$ so that if $|V|<V^\ast$ and $Z>Z^\ast$, then all but possibly one of the eigenvalues of $T_C(V)$ has negative real part. We also know that when $V=0$, this one eigenvalue is zero. Therefore, we only need to show that for $0<|V|<V^\ast$, this eigenvalue has negative real part.

		 Since $T_C$ depends on $V$, both explicitly, and through $m_T$ and $\phi_T$, we write $T_C=T_C(V)$. For the operator $T_C(V)$, the parameter $P=P_T(V)$ is given by Theorem \ref{thm:bifurcation}. We also consider the linearization $S_C(P)$ of $F$ about $m=1$ with arbitrary $P>0$. We will make use of Corollary 1.13 and Theorem 1.16 in \cite{CraRab1973} which from which we conclude the following: 
		\begin{itemize}
			\item There exists neighborhoods $U_1,U_2\subset\mathbb R$ of $0$ and $P_0\in\mathbb R$  respectively and smooth functions $\lambda:U_1\to \mathbb R$ and $\mu:U_2\to\mathbb R$ such that $\lambda(V)$ is an eigenvalue of $T_C(V)$ and $\mu(P)$ is an eigenvalue of $T_C(P)$, and $\lambda(0)=\mu(P_0)=0$.
			\item $\lambda$ and $\mu$ satisfy:
			\begin{equation}\label{eq:eval_limit}
				-\mu'(P_0)\lim_{V\to 0}\frac{VP_T'(V)}{\lambda(V)}=1.
			\end{equation}
		\end{itemize}
		 By Lemma \ref{lem:mostly_negative_evs}, $\lambda(0)=0$ is the largest eigenvalue of $T_C(0)$. Since $T_C$ depends smoothly on $V$, so does $\lambda(V)$. Therefore, for small $V$, $\lambda(V)$ is the eigenvalue of $T_C(V)$ with the largest real part. Moreover, for small $V$, $\lambda(V)$ has the same sign as $-VP'(V)\mu'(P_0)$. From Proposition \ref{lem:P0_asymptotic_form}, $P_T'(0)=0$. For similar reasons, $\lambda'(0)=0$. So after two applications of L'H\^opital's rule on \eqref{eq:eval_limit}, we obtain $\lambda(V)=\frac{1}{2}\lambda''(0)V^2+O(V^3)$ and
		 \begin{equation}
		 	\lambda''(0)=-2P_T''(0)\mu'(P_0).
		 \end{equation}
		 Therefore, if $P_T''(0)\mu'(P_0)>0$, then there exists $V^\ast>0$ such that if $0<|V|<V^\ast$, then $\lambda(V)<0$. We will show that for sufficiently large $Z$, both $P_T''(0)$ and $\mu'(P_0)$ are positive, thus proving the desired result.
		 
		 First we show that $\mu'(P_0)$ is positive. The eigenvalue equation satisfied by $\mu(P)$ is 
		 \begin{equation}\label{eq:SS_operator_eigenequation}
		 	u''-\phi''=\mu(P)u,\quad -Z\phi''+\phi=Pu,
		 \end{equation}
		 where $m(\pm 1/2)=0$ and $\phi$ satisfies periodic boundary conditions. We write $P=P_0+\varepsilon$ for some small $\varepsilon$, and expand $m$, $\phi$, and $\mu$ in $\varepsilon$:
		 \begin{align}
		 	u&=u_0+\varepsilon u_1+O(\varepsilon^2)\\
		 	\phi&=\phi_0+\varepsilon \phi_1+O(\varepsilon^2)\\
		 	\mu&=\mu_1\varepsilon+O(\varepsilon^2).
		 \end{align}
		 Observe that $\mu_1=\mu'(P_0)$. Solving the zero order in $\varepsilon$ equation, we find $u_0$ and $\phi_0$ up to a multiplicative constant:
		 \begin{equation}
		 	u_0=\frac{x}{P_0-1}-\frac{1}{2}\frac{P_0}{P_0-1}\csc\left(\frac{\sqrt{P_0-1}}{\sqrt{Z}}\right) \sin \left(\frac{\sqrt{P_0-1} x}{\sqrt{Z}}\right), \quad \phi_0=u_0+x.
		 \end{equation}
		 Observe that since $P_0$ and $Z$ satisfy \eqref{eq:bifurcation_condition},
		 \begin{equation}
		 	\csc\left(\frac{\sqrt{P_0-1}}{\sqrt{Z}}\right)=\sqrt{\frac{P_0^3-P_0^2+4 Z}{\left(P_0-1\right) P_0^2}}.
		 \end{equation}
		 In first order, the \eqref{eq:SS_operator_eigenequation} becomes
		 \begin{equation}
		 u_1''-\phi_1''=\mu_1u_0,\quad -Z\phi_1''+\phi_1=P_0u_1+u_0.
		 \end{equation}
		 Write $\phi_1=\psi_1+\phi_0/P_0$ where $\psi_1$ solves $-Z\psi_1''+\psi_1=P_0u_1$. Thus, we may write the first order equation as
		 \begin{equation}
		 	u_1''-\psi_1''=\mu_1u_0-\frac{1}{Z}u_0+\frac{\phi_0}{P_0Z}.
		 \end{equation}
		 Since the operator $u_1\mapsto u_1''-\psi_1''$ (which is $S_C(P_0)$) is self adjoint, the right hand side must be orthogonal to the kernel of the operator, which is spanned by $u_0$. Thus, $\mu_1$ solves:
		 \begin{equation}
		 \int_{-1/2}^{1/2}\left(\mu_1u_0-\frac{1}{Z}u_0+\frac{\phi_0}{P_0Z}\right)u_0\,dx=0.
		 \end{equation}
		 Computing the integral and solving for $\mu_1$, we obtain
		 \begin{equation}
		 \mu_1=\frac{3 \left(P_0-1\right) \left(P_0^2-12 Z\right)}{P_0 Z \left(3 P_0^2-60 Z+2\right)}.
		 \end{equation}
		 Using Lemma \ref{lem:P0_asymptotic_form}, we obtain an asymptotic form for $\mu_1$ in large $Z$:
		 \begin{equation}
		 \mu_1=\frac{1}{Z}+O(1/Z^2).
		 \end{equation}
		 Thus, for sufficiently large $Z$, $\mu'(P_0)=\mu_1>0$.
		 
		 Lemma \ref{lem:TW_asymptotic_form} gives the value of $P_2$. In large $Z$, this expands as
		 \begin{equation}
		 	P_2=\frac{\pi^2}{48}Z+O(1).
		 \end{equation}
		Thus, for large $Z$, $P_2>0$. Thus,
		\begin{equation}
			\lambda(V)=-\frac{\pi^2}{24}V^2+O(\frac{V^4}{Z}),
		\end{equation}
		so for large $Z$ and small $V$, the largest real part of the eigenvalues of $T_C(V)$ is negative.		
	\end{proof}
	
	We conclude with a technical lemma used in the proof of Lemma \ref{lem:mostly_negative_evs}:	
	\begin{lemma}\label{lem:inner_pruduct_perturbation}
		Suppose $f:[-1/2,1/2]\to \mathbb R$ is $C^2$. Let $\mathcal B=\{v_1,v_2,v_3,\cdots\}$ where $v_n(x)=\sin(n\pi x)$ for $n$ odd, and $v_n(x)=\cos(n\pi x)$ for $n$ even. Then there exists $C>0$ such that
		\begin{equation}
			\sum_{n=1}^\infty|\langle v_m,fv_n\rangle_{L^2}|<  C\Vert f\Vert_{L^1}\quad\text{and}\quad \sum_{m=1}^\infty|\langle v_m,fv_n\rangle_{L^2}|<  C\Vert f\Vert_{L^1}.
		\end{equation}
	\end{lemma}
	\begin{proof}
		Decompose $f$ as a Fourier series: $f=\sum_{k=1}^\infty a_kv_k$. Since $f$ is $C^2$-smooth, $|a_k|<\Vert f''\Vert_{L^1}/k^2$ Then we can use some product-to-sum trigonometric identities to see that
		\begin{equation}
			fv_n=\sum_{k=1}^\infty a_k v_kv_n=\sum_{k=1}^\infty \frac{a_k}{2}(r_{n,k} v_{k+n}+s_{n,k} v_{|k-n|}),
		\end{equation}
		where the the coefficeints $r_{n,k}$ and $s_{n,k}$ are either $1$ or $-1$ and are determined by the parities of $n$ and $k$. The sign of each coefficient is not important, so we do not endeavor to give them explicitly. Thus,
		\begin{align}
			\sum_{n=1}^\infty|\langle v_m,fv_n\rangle_{L^2}|&=\sum_{n=1}^\infty\left|\sum_{k=1}^\infty
			\frac{a_k}{2}\langle v_m,r_{n,k} v_{k+n}+s_{n,k}v_{|k-n|}\rangle\right|\\
			&\leq \sum_{n=1}^\infty\sum_{k=1}^\infty\frac{|a_k|}{2}(|\langle v_m,v_{k+n}\rangle|+|\langle v_m, v_{|k-n|}\rangle|)\\
			&=\frac{1}{4}\sum_{n=1}^\infty |a_{n+m}|+|a_{|n-m|}|\\
			&\leq\frac{3}{4}\sum_{n=1}^\infty |a_n|\\
			&\leq\frac{3\Vert f''\Vert_{L^1}}{4}\sum_{n=1}^\infty\frac{1}{n^2}\\
			&=\frac{\pi^2}{8}\Vert f''\Vert_{L^1}.
		\end{align}
		Thus the result for the sum over $n$ holds. The proof for the sum over $m$ is identical.
	\end{proof}

	\paragraph{Condition (i): $T_C$ generates a strongly continuous semigroup, see Appendix C.}
	Here we show that the linearized operator $T_C$ defined by \eqref{eq:SL_TW_linearization} generates a strongly continuous semigroup. We will make use to of the Hille-Yosida Theorem \cite{EngNagBre2000}. We will first prove a supporting proposition.
	\begin{proposition}\label{prop:strong_continuity_bound_1}
		There exists $V^\ast,Z^\ast,\lambda_0>0$ such that if $|V|<V^\ast$ and $Z>Z^\ast$, then all eigenvalues of for all $\lambda>0$ and $u\in \tilde X^2_C$
		\begin{equation}
			(\lambda-\lambda_0)\left\Vert u\right\Vert_{L^2}\leq \left\Vert (\lambda I-T_C)u\right\Vert_{L^2}.
		\end{equation}
	\end{proposition}
	\begin{proof}
		We calculate the norm via the inner product:
		\begin{align*}
			\left\Vert (\lambda I-T_C)u\right\Vert_{L^2}^2&=\left\langle (\lambda I-T_C)u,(\lambda I-T_C)u\right\rangle_{L^2}\\
			&=\lambda^2\left\Vert u\right\Vert_{L^2}^2+\left\Vert T_Cu\right\Vert_{L^2}^2-2\lambda\left\langle u,T_Cu\right\rangle_{L^2}.
		\end{align*}
		Observe that
		\begin{equation}
			\langle u, T_C u\rangle_{L^2}\leq -\Vert u'\Vert_{L^2}^2+\Vert(\phi'(1/2)-\phi)m_T'\Vert_{L^2}^2+\Vert(V-\phi_T)u'\Vert_{L^2}^2+\Vert m_T\phi''\Vert_{L^2}^2+\Vert u\phi_T''\Vert_{L^2}^2.
		\end{equation}
		There exists $C,V^\ast>0$ so that $|V|<V^\ast$ so that (after applying the Poincar\'e inequality):
		\begin{equation}
			\begin{split}
				\Vert(\phi'(1/2)-\phi)m_T'\Vert_{L^2}^2&<C|V|\Vert u\Vert_{L^2}^2\\
				\Vert(V-\phi_T')u'\Vert_{L^2}^2&<C|V|\Vert u'\Vert_{L^2}^2\\
				\Vert m_T\phi''\Vert_{L^2}^2&\leq C\Vert u\Vert_{L^2}^2\\
				.\Vert u\phi_T''\Vert_{L^2}^2&\leq C|V|\Vert u\Vert_{L^2}^2.
			\end{split}
		\end{equation}
		Assume $V^\ast$ is sufficiently small that $C|V|<1/2$. 
			\begin{equation}
			\langle u, T_C u\rangle_{L^2}\leq -\frac{1}{2}\Vert u'\Vert_{L^2}^2+(1+C)\Vert u\Vert_{L^2}^2.
		\end{equation}
		Let $\lambda_0=2(1+C)$ so that $\langle u, T_C u\rangle_{L^2}\leq(\lambda_0/2)\Vert u\Vert_{L^2}^2$. Thus,
		\begin{align}
			\Vert (\lambda I-T_C)u\Vert_{L^2}^2&\geq \lambda^2\Vert u\Vert_{L^2}-\lambda\lambda_0\Vert u\Vert_{L^2}^2=(\lambda^2-\lambda\lambda_0)\Vert u\Vert_{L^2}^2.
		\end{align}
		If $\lambda>\lambda_0$, then $\lambda^2-\lambda\lambda_0\geq \lambda^2-2\lambda\lambda_0+\lambda_0^2=(\lambda-\lambda_0)^2.$ Therefore,
		\begin{equation}
			\Vert (\lambda I-T_C)u\Vert_{L^2}\geq (\lambda-\lambda_0)\Vert u\Vert_{L^2}.	
		\end{equation}		
	\end{proof}
	
	We recall the definition of a closed operator:
	\begin{definition}\label{def:closed_operator}
		Let $X$ and $Y$ be Banach spaces and let $B:D(B)\subset X\to Y$ be a linear operator. Then $B$ is \textit{closed} if for every sequence $(x_n)$ converging to some $x\in X$ such that $Bx_n$ converges to $y\in Y$, it follows that $x\in D(B)$ and $Bx=y$.
	\end{definition}
	An operator is closed if its resolvent $(\lambda I-B)^{-1}$ exists and is bounded for at least one value of $\lambda\in\mathbb C$. By Theorem \ref{thm:negative_eigenvalues}, the resolvent set of $T_C$ is non-empty, and by Proposition \ref{prop:TW_compact_inverse}, the resolvent is compact (and thus bounded) whenever it exists. Therefore, $T_C$ is a closed operator. Thus, we may prove the main result of this section:
	\begin{proposition}\label{prop:strongly_continuous}
		There exists $V^\ast>0$ such that if $|V|<V^\ast$, then $A$ generates a strongly continuous semigroup.
	\end{proposition}
	\begin{proof}
		We appeal the the Hille-Yosida Theorem \cite{EngNagBre2000}, which states that if $T_C:X\to Y$ is a closed, densely defined operator and if there exists $\lambda_0>0$ such that
		\begin{equation}\label{eq:HY-condition}
			\Vert (\lambda I-T_C)^{-n}\Vert_{L^2}\leq \frac{1}{(\lambda-\lambda_0)^n},
		\end{equation}
		then $T_C$ generates a strongly continuous semigroup.
		
		It is clear to see that \eqref{eq:HY-condition} is satisfied due to Proposition \ref{prop:strong_continuity_bound_1}. Therefore, the hypotheses of the Hille-Yosida theorem are satisfied for sufficiently small $V^\ast$, so the result holds.
	\end{proof}
	Since $T_C$ generates a strongly continuous semigroup, the first condition of the Grearhart-Pr\"uss-Griener Theorem is satisfied.
	
	\paragraph{Condition (iii): the resolvent of $T_C$ is uniformly bounded, see appendix C.}
	
Now we prove that the resolvent of $T_C$ is uniformly bounded for complex numbers with positive real part. Then we formally establish linear stability in Theorem \ref{thm:linear_stability}.
	
	\begin{proposition}\label{prop:resolvent_bounded}
		There exist $V^\ast,Z^\ast,\Gamma>0$ such that if $0<|V|<V^\ast$ and $Z>Z^\ast$, then the resolvent $\Vert(\lambda I-T_C)^{-1}\Vert<\Gamma$ for all $\lambda\in\mathbb C$ with $\text{Re}\,\lambda>0$.
	\end{proposition}
	\begin{proof}
		Existence of the resolvent $(\lambda I-T_C)^{-1}$ for all $\lambda$ with $\text{Re}\,\lambda>0$ is established in Theorem \ref{thm:negative_eigenvalues}. Assume, to the contrary, that there exists a sequence $(\lambda_k)_{k=1}^\infty\subset\mathbb C$ such that $\text{Re}\,\lambda_k>0$ for each $k$ and
		\begin{equation}
			\Vert (\lambda_k I-T_C)^{-1}\Vert_{L^2}>k.
		\end{equation}
		Then for each $k$, there exist $v_k\in \tilde X^2_C$ and $w_k\in L^2(-1/2,1/2)$ such that $(\lambda_k I-T_C)v_k=w_k$, $\Vert v_k\Vert_{L^2}=1$, and $\Vert w_k\Vert_{L^2}<1/k$. We shall consider two cases: (i) the sequence $(\lambda_k)$ is bounded, and (ii) $(\lambda_k)$ is unbounded. We will show that in each case, we arrive at a contradiction.
		
		\begin{itemize}
			\item[(i)] If the sequence $(\lambda_k)$ is bounded, then it has a subsequence also called $(\lambda_k)$ which converges to some $\lambda\in\mathbb C$ with $\text{Re}\,\lambda\geq 0$. By Theorem \ref{thm:negative_eigenvalues}, $\lambda$ is in the resolvent set of $T_C$. Recall the \textit{first resolvent identity} \cite{DunSch1988} from which we conclude that for each $k$,
			\begin{equation}
				(\lambda I-T_C)^{-1}-(\lambda_k I-T_C)^{-1}=(\lambda-\lambda_k)(\lambda I-T_C)^{-1}(\lambda_k I-T_C)^{-1}.
			\end{equation}
			We calculate:
			\begin{align*}
				\Vert v_k\Vert_{L^2}&=\Vert (\lambda_kI-T_C)^{-1}w_k\Vert_{L^2}\\
				&\leq \left\Vert -\left[(\lambda I-T_C)^{-1}-(\lambda_k I-T_C)^{-1}\right]w_k\right\Vert_{L^2}+\Vert (\lambda I-T_C)^{-1}w_k\Vert_{L^2}\\
				&\leq\left\Vert(\lambda_k-\lambda)(\lambda I-T_C)^{-1}(\lambda_k I-T_C)^{-1}w_k\right\Vert_{L^2}+\Vert w_k\Vert_{L^2}\Vert (\lambda I-T_C)^{-1}\Vert\\
				&\leq|\lambda_k-\lambda|\left\Vert(\lambda I-T_C)^{-1}v_k\right\Vert_{L^2}+\Vert w_k\Vert_{L^2}\Vert (\lambda I-T_C)^{-1}\Vert\\
				&\leq\left(|\lambda_k-\lambda|\Vert v_k\Vert_{L^2}+\Vert w_k\Vert_{L^2}\right)\Vert(\lambda I-T_C)^{-1}\Vert_{L^2}.
			\end{align*}
			Since $|\lambda_k-\lambda_0|,\Vert w_k\Vert_{L^2}\to 0$ and $\Vert v_k\Vert_{L^2}$ is bounded, we conclude that $\Vert v_k\Vert_{L^2}\to 0$, a contradiction. Therefore, $(\lambda_k)$ is not bounded.
			
			\item[(ii)] If the sequence $(\lambda_k)$ is unbounded, then it has a subsequence also called $(\lambda_k)$ such that $\lambda_k\to\infty$. There exists corresponding sequences $(v_k)$ and $(w_k)$ such that
			\begin{equation}
				w_k=(\lambda_kI-T_C)v_k,\quad \Vert v_k\Vert_{L^2}=1,\quad \Vert w_k\Vert_{L^2}\leq 1/k.
			\end{equation}
			We calculate the inner product
			\begin{equation}\label{eq:inner_product}
				\begin{split}
					\langle w_k,v_k\rangle_{L^2}&=\lambda_k+\Vert v_k'\Vert_{L^2}+\int_{-1/2}^{1/2}(V-\phi_T')v_k'\bar v_k\,dx\\
					&+\int_{-1/2}^{1/2}m_T'(\phi_k'(1/2)-\phi_k')\bar v_k\,dx-\int_{-1/2}^{1/2}\phi_T''|v_k|^2\,dx-\int_{-1/2}^{1/2}m_T\phi_k''\bar v_k\,dx
				\end{split}
			\end{equation}
			Since $(v_k)$ is $L^2$-bounded, by Proposition \ref{prop:elliptic_estimates} in Appendix A, the last three integrals in \eqref{eq:inner_product} are uniformly bounded:
			\begin{equation}
				\left|\int_{-1/2}^{1/2}m_T'(\phi_k'(1/2)-\phi_k')\bar v_k\,dx-\int_{-1/2}^{1/2}\phi_T''|v_k|^2\,dx-\int_{-1/2}^{1/2}m_T\phi_k''\bar v_k\,dx\right|<C
			\end{equation}
			for some $C>0$ independent of $k$.
			
			Taking the real part of \eqref{eq:inner_product}, we find using the Cauchy-Schwartz inequality and the Poincar\'e inequality that
			\begin{equation}\label{eq:inner_product_real_part}
				\text{Re}\,\langle w_k,v_k\rangle\geq \text{Re}\,\lambda_k+\Vert v_k'\Vert_{L^2}-\frac{1}{\pi}\Vert V-\phi_T'\Vert_{L^\infty}\Vert v_k'\Vert_{L^2}^2-C.
			\end{equation}
			Assuming $V^\ast$ is sufficiently small that if $|V|<V^\ast$, then $\Vert V-\phi_T'\Vert_{L^\infty}<\pi$, we conclude that $\text{Re}\,\langle w_k,v_k\rangle\geq \text{Re}\,\lambda_k-C$. On the other hand, $\text{Re}\,\langle w_k,v_k\rangle\leq |\langle w_k,v_k\rangle|<1/k$. Since $\text{Re}\,\lambda_k>0$, we conclude that $(\text{Re}\,\lambda_k)$ is bounded. Furthermore, since all terms in \eqref{eq:inner_product_real_part} have been shown to be bounded except those involving  $\Vert v_k'\Vert$, we conclude that $(v_k')$ must be bounded as well.
			
			Now taking the imaginary part of \eqref{eq:inner_product}, we find that
			\begin{equation}\label{eq:inner_product_imaginary_part}
				\text{Im}\,\langle w_k,v_k\rangle\geq \text{Im}\,\lambda_k-\frac{1}{\pi}\Vert V-\phi_T'\Vert_{L^\infty}\Vert v_k'\Vert_{L^2}^2-C.
			\end{equation}
			Once again, all terms in this equation are known to be bounded in $k$ except $\text{Im}\,\lambda_k$, so we conclude that $(\text{Im}\,\lambda_k)$ is bounded also, a contradiction.
		\end{itemize}
		
		Since $(\lambda_k)$ can be neither bounded nor unbounded, we conclude that no such sequence $(\lambda_k)$ can exist, and so $(\lambda I-T_C)^{-1}$ is uniformly bounded. That is, there exists $\Gamma>0$ such that
		\begin{equation}
			\Vert (\lambda I-T_C)^{-1}\Vert<\Gamma.
		\end{equation}
	\end{proof}
	
	Now that we have in place all the results proving the conditions of the GPG theorem, we may apply it to prove linear stability.
	\begin{theorem}\label{thm:linear_stability}
		There exist $V^\ast,Z^\ast,\Gamma,\sigma>0$ such that if $|V|<V^\ast$ and $Z>Z^\ast$ then $T_C$ generates a strongly continuous semigroup $\{S(t):t\geq 0\}$ satisfying
		\begin{equation}\label{eq:exponential_decay_of_semigroup}
			\Vert S(t)\Vert<\Gamma e^{-\sigma t}.
		\end{equation}
	\end{theorem}
	\begin{proof}
		We need to satisfy the three hypotheses of the GPG theorem \ref{thm:GPG}, see Appendix C. We have checked that indeed: 
		\begin{itemize}
			\item (i) is satisfied for sufficiently small $V^\ast$ due to Proposition \ref{prop:strongly_continuous}
			
			\item (ii) is satisfied for sufficiently small $V^\ast$ and sufficiently large $Z^\ast$ due to Theorem \ref{thm:negative_eigenvalues}.
			
			\item (iii) is satisfied due to Proposition \ref{prop:resolvent_bounded}.
		\end{itemize}
		Thus, the desired result holds.
	\end{proof}
		
	
We can now  prove that traveling wave solution $m_T$ to model $C$ are asymptotically stable. Specifically, we will prove the following theorem:
	\begin{theorem}\label{thm:nonlinear_stability}
		Fix $V\in\mathbb R$ and $Z>0$, and let $m(x,t)$ be a solution to  \eqref{eq:model_C} with $m(0,x)=m_0(x)$. Let $m_T$ denote the traveling wave solution to \eqref{eq:model_C} with velocity $V$. There exist $V^\ast, Z^\ast>0$ independent of $V$ and $Z$ and $\varepsilon>0$ depending on $V$ and $Z$ such that if $|V|<V^\ast$, $Z>Z^\ast$ and
		\begin{equation}
			\Vert m_0-m_T\Vert_{H^1}<\varepsilon,
		\end{equation}
		then
		\begin{equation}
			\lim_{t\to\infty}\Vert m(\cdot,t)-m_T\Vert_{L^2}=0.
		\end{equation}
	\end{theorem}
		
	To prove Theorem \ref{thm:nonlinear_stability}, we follow the same strategy as proving Theorem \ref{thm:stability_of_SS}. We decompose $F_C$ as a sum of its linearization $T_C$ about $m_T$ and its ``nonlinear part''. Since the nonlinearity in $F_C$ is quadratic (that is, the Keller-Segel term $(m\phi')'$), the nonlinear part about the traveling wave $m=m_T$ is the same as the nonlinear part about the stationary state $m=1$.  Thus, 
	\begin{equation}
		F_C(m_T+u)=T_Cu+\Psi(u),
	\end{equation}
	where $\Psi$ is given by \eqref{eq:Psi}. Thus, we may directly apply Proposition \ref{lem:Psi_estimate}. We will prove a version of \ref{lem:w_derivative_bounded} for $T_C$ showing that the linear part of $F_C$ dominates the nonlinear part in a neighborhood of $m_T$. Finally, the proof of Theorem \ref{thm:nonlinear_stability} is identical to the proof of Theorem \ref{thm:stability_of_SS}.
	
	\begin{lemma}\label{lem:u_derivative_bounded_TW}
		Let $T,\delta>0$ and let $u$ be a solution to
		\begin{equation}
			\begin{cases}
				\partial_tu=T_Cu+\Psi(u) & -1/2<x<1/2,\;0<t<T\\
				u'=0 & x=\pm 1/2,\;t>0.
			\end{cases}
		\end{equation}
		There exist $V^\ast,U^\ast>0$ such that if $\Vert u'(0,\cdot)\Vert_{L^2}<\delta$, $|V|<V^\ast$, and $\Vert u(\cdot,t)\Vert_{L^2}<U^\ast$ for all $0\leq t\leq T$, then
		\begin{equation}
			\Vert u'(\cdot,t)\Vert_{L^2}\leq \delta
		\end{equation}
		for all $0\leq t\leq T$.
	\end{lemma}
	\begin{proof}
		Write the evolution equation \eqref{eq:model_C} as
		\begin{equation}
			\partial_t  u-u''=Bu+\Psi(m).
		\end{equation}
		Where $B$ is defined by
		\begin{equation}\label{eq:B2}
			Bu=(\phi'(1/2)-\phi')m_T'+(V-\phi_T')u'-m_T\phi''-u\phi_T'',\quad \begin{cases}-Z\phi''+\phi=Pu\\
				\phi(-1/2)=\phi(1/2)\\ \phi'(-1/2)=\phi'(1/2).\end{cases}
		\end{equation}
		Now square both sides and integrate to obtain
		\begin{align}
			\Vert Bu+\Psi(u)\Vert_{L^2}^2&=\int_{-1/2}^{1/2}\Psi^2(u)\,dx\\
			&=\int_{-1/2}^{1/2} (\partial_t  u)^2-2(\partial_t  u)u''+(u'')^2\,dx\\
			&=\Vert\partial_t  u\Vert_{L^2}^2+2\int_{-1/2}^{1/2}(\partial_t  u')u'\,dx+\Vert m''\Vert_{L^2}^2\\
			&=\Vert\partial_t  u\Vert_{L^2}^2+\frac{d}{dt}\Vert u'\Vert_{L^2}^2+\Vert u''\Vert_{L^2}^2.
		\end{align}
		Thus,
		\begin{align}
			\frac{d}{dt}\Vert u'\Vert^2_{L^2}&\leq\Vert Bu+\Psi(u)\Vert_{L^2}^2-\Vert u''\Vert_{L^2}^2\\
			&\leq2\Vert Bu\Vert_{L^2}^2+2\Vert\Psi(u)\Vert_{L^2}^2-\Vert u''\Vert_{L^2}^2.
		\end{align}
		From Lemma \ref{lem:Psi_estimate}, there exists $C_1$ independent of $u$, $V$, and $Z$ such that
		\begin{equation}
			\Vert\Psi(u)\Vert_{L^2}\leq C_1\Vert u\Vert_{L^2}\Vert u\Vert_{H^1}.
		\end{equation}
		Observe that due to \ref{prop:elliptic_estimates}, if $|V|<V^\ast$ is small enough, there exist $C_2$, $C_3$ depending only on $Z$ such that
		\begin{equation}
			\Vert Bu\Vert_{L^2}\leq C_2V^\ast \Vert u'\Vert_{L^2}+C_3\Vert u\Vert_{L^2}.
		\end{equation}
		Since $\int_{-1/2}^{1/2}u\,dx=0$ and $u'(\pm 1/2,t)=0$, we may apply the Poincar\'e inequality to both $u$ and $u'$ with a Poincar\'e constant of $\pi$:
		\begin{equation}\label{eq:Poincare_m_prime}
			\pi\Vert u\Vert_{L^2}\leq\Vert u'\Vert_{L^2}\quad\text{and}\quad\pi\Vert u'\Vert_{L^2}\leq\Vert u''\Vert_{L^2}.
		\end{equation}
		Thus,
		\begin{align}
			\frac{d}{dt}\Vert u'\Vert^2_{L^2}&\leq 2C_1^2\Vert u\Vert_{L^2}^2\Vert u\Vert_{H^1}^2+4C_2^2(V^\ast)^2\Vert u'\Vert_{L^2}^2+4C_3^2\Vert u\Vert_{L^2}^2-\Vert u''\Vert_{L^2}^2\\
			&\leq2C_1^2\left(1+\frac{1}{\pi}\right)^2\Vert u\Vert_{L^2}^2\Vert u'\Vert_{L^2}+4C_2^2(V^\ast)^2\Vert u'\Vert_{L^2}^2+4C_3^2\Vert u\Vert_{L^2}^2-\Vert u''\Vert_{L^2}^2\\
			&\leq-\left(\pi^2-4C_1^2\Vert u\Vert_{L^2}^2-4C_2^2(V^\ast)^2\right)\Vert u'\Vert_{L^2}^2+4C_3^2\Vert u\Vert_{L^2}^2.
		\end{align}
		Without loss of generality, we may assume that
		\begin{equation}
			V^\ast\leq\frac{\pi}{4C_2}\quad\text{and}\quad U^\ast\leq\frac{\pi}{4C_1}.
		\end{equation}
		Then, if $\Vert u\Vert_{L^2}\leq U^\ast$ for all $0<t<T$,
		\begin{equation}
			\frac{d}{dt}\Vert u'\Vert^2_{L^2}\leq -R_1\Vert u'\Vert_{L^2}^2+R_2
		\end{equation}
		where
		\begin{equation}
			R_1=\pi^2-4C_1^2(U^\ast)^2-4C_2^2(V^\ast)^2\geq\frac{\pi^2}{2}\quad\text{and}\quad R_2=4C_3^2(U^\ast)^2.
		\end{equation}
		
		We now introduce a new variable:
		\begin{equation}
			q(t)=\Vert u'(\cdot,t)\Vert_{L^2}^2-\frac{R_2}{R_1}.
		\end{equation}
		Then $q$ satisfies $q'\leq -R_1q.$ By Gr\"onwall's inequality,
		\begin{equation}
			q(t)\leq q(0)e^{-R_1t}.
		\end{equation}
		We conclude that if $q(0)<0$, then $q(t)<0$ for all $t>0$. Thus, if $|V|<V^\ast$ and $\Vert u\Vert_{L^2}\leq U^\ast$, and if
		\begin{equation}
			\Vert u'(\cdot,0)\Vert_{L^2}<\sqrt{\frac{R_2}{R_1}},\quad \text{then}\quad
			\Vert u'(\cdot,t)\Vert_{L^2}<\sqrt{\frac{R_2}{R_1}}
		\end{equation}
		for all $t>0$. Letting $U^\ast$ be sufficiently small that
		\begin{equation}
			\sqrt{\frac{R_2}{R_1}}\leq\frac{\sqrt{8}}{\pi}C_3 U^\ast<\delta,
		\end{equation}
		the desired result holds.
	\end{proof}
	
	With Lemma \ref{lem:u_derivative_bounded_TW} in place, we may duplicate the proof of Theorem \ref{thm:stability_of_SS} in order to prove the nonlinear stability of traveling waves via Theorem \ref{thm:nonlinear_stability}.
	
	\section*{Acknowledgments}
	
	We  thank V. Rybalko for many helpful discussions on NSA and the relevance in this case of the GPG theorem. We also thank O. Krupchytskyi for his feedback on the proofs and mathematical techniques used in this paper. Finally, we thank J.-F. Joanny,  J. Casademunt and P. Recho for discussing the physical aspects of the model and the subtlety of stability in the problems with NSA. L. B. was supported by the National Science Foundation grants DMS-2005262 and DMS-2404546. A. S. was also partially supported by the same National Science Foundation grant DMS-2005262.  L.T.
  acknowledges the support   under the French grants ANR-17-CE08-0047-02,  ANR-21-CE08-MESOCRYSP  
and the European grant ERC-H2020-MSCA-RISE-2020-101008140.
	
	\section{Appendix A }
	
Here we show the   Proposition which  controls the solution $\phi$ to \eqref{eq:SL_phi}, \eqref{eq:SL_phi_BC_1}-\eqref{eq:SL_phi_BC_2}.
		\begin{proposition}\label{prop:elliptic_estimates}
			Let $u\in L^0([-1/2,1/2])$. Then there exists a unique solution $\phi\in W^{2,p}(-1/2,1/2)$ for any $1\leq p\leq\infty$ satisfying $-Z\phi''+\phi=Pu$ with periodic boundary conditions on $(-1/2,1/2)$. Moreover, $\phi$ satisfies the following for any $1\leq p\leq\infty$:
			\begin{itemize}
				\item $\Vert \phi\Vert_{L^p}\leq P\Vert u\Vert_{L^p}$,
				\item $\Vert\phi'\Vert_{L^\infty}\leq \frac{P}{2Z}\Vert u\Vert_{L^2}$,
				\item $\Vert \phi''\Vert_{L^p}\leq \frac{2P}{Z}\Vert u\Vert_{L^p}$.
			\end{itemize}
		\end{proposition}
		\begin{proof}
			The solution $\phi$ can be calculated explicitly using a Green's function:
			\begin{equation}
				\phi(x)=\frac{P}{2\sqrt{Z}\sinh\left(\frac{1}{2\sqrt{Z}}\right)}\int_{-1/2}^{1/2}G(x,y)u(y)\,dy,\quad G(x,y)=\begin{cases}
					\cosh\left(\frac{1/2+(y-x)}{\sqrt{Z}}\right) & y<x\\
					\cosh\left(\frac{1/2+(x-y)}{\sqrt{Z}}\right) & y>x.
				\end{cases}.
			\end{equation}
			By Young's Integral inequality \cite{Rus1977}, $\Vert\phi\Vert_{L^p}\leq C\Vert u\Vert_{L^p}$ where 
			\begin{equation}
				C=\sup_{|x|\leq 1/2}\frac{P}{2\sqrt{Z}\sinh\left(\frac{1}{2\sqrt{Z}}\right)}\int_{-1/2}^{1/2}|G(x,y)|\,dy=\sup_{|y|\leq 1/2}\frac{P}{2\sqrt{Z}\sinh\left(\frac{1}{2\sqrt{Z}}\right)}\int_{-1/2}^{1/2}|G(x,y)|\,dx.
			\end{equation}
			We calculate
			\begin{align}
				\frac{P}{2\sqrt{Z}\sinh\left(\frac{1}{2\sqrt{Z}}\right)}&\int_{-1/2}^{1/2}|G(x,y)|\,dy\\
				&=\frac{P}{2\sqrt{Z}\sinh\left(\frac{1}{2\sqrt{Z}}\right)}\left(\int_{-1/2}^x\cosh\left(\frac{1/2+y-x}{\sqrt{Z}}\right)\,dy+\int_x^{1/2}\cosh\left(\frac{1/2+x-y}{\sqrt{Z}}\right)\,dy\right)\\
				&=\frac{P}{2\sinh\left(\frac{1}{2\sqrt{Z}}\right)}\left(\sinh\left(\frac{1}{2\sqrt{Z}}\right)+\sinh\left(\frac{x}{\sqrt{Z}}\right)-\sinh\left(\frac{x}{\sqrt{Z}}\right)+\sinh\left(\frac{1}{2\sqrt{Z}}\right)\right)\\
				&=P.
			\end{align}
			We conclude that $\Vert \phi\Vert_{L^p}\leq P\Vert u\Vert_{L^p}$. Next, since $G$ is continuous and differentiable in $x$ except where $x=y$,
			\begin{equation}
				\phi'(x)= \frac{P}{2\sqrt{Z}\sinh\left(\frac{1}{2\sqrt{Z}}\right)}\int_{-1/2}^{1/2}\frac{d}{dx}G(x,y)u(y)dy.
			\end{equation}
			Therefore, using H\"older's inequality for $p$ and its H\"older conjugate $q$,
			\begin{equation}
				\Vert\phi'\Vert_{L^\infty}\leq \frac{P}{2\sqrt{Z}\sinh\left(\frac{1}{2\sqrt{Z}}\right)}\Vert u\Vert_{L^p}\sup_{|x|\leq 1/2}\left\Vert\frac{d}{dx}G(x,\cdot)\right\Vert_{L^q}.
			\end{equation}
			Since $|\frac{d}{dx}G(x,y)|\leq \frac{1}{\sqrt{Z}}\sinh\left(\frac{1}{2\sqrt{Z}}\right)$, we conclude that $\Vert \phi'\Vert_{L^\infty}\leq \frac{P}{2Z}\Vert u\Vert_{L^p}.$ Finally, since $\phi''=-\frac{P}{Z}u+\frac{1}{Z}\phi$, we have
			\begin{equation}
				\Vert \phi''\Vert_{L^p}\leq \frac{P}{Z}\Vert u\Vert_{L^p}+\frac{1}{Z}\Vert\phi\Vert_{L^p}\leq\frac{2P}{Z}\Vert u\Vert_{L^p}.
			\end{equation}
		\end{proof}
		
		\section{Appendix B}		
		 Here we formulate  for convenience the Crandall-Rabinowitz (CR) theorem \cite{CraRab1971}.		
		\begin{theorem}
			Let $X$ and $Y$ be Banach spaces, and let $\mathcal F:X\times \mathbb R\to Y$ be an operator with the following properties:
			\begin{itemize}
				\item $\mathcal F(0,t)=0$ for all $t$.
				\item $D_x\mathcal F$, $D_t\mathcal F$, and $D_{xt}\mathcal F$ exist and are continuous.
				\item The dimension of the null space and co-dimension of the range of $D_x\mathcal F(0,0)$ are both 1.
				\item If $x_0\neq 0$ is in the null space of $D_x\mathcal F(0,0)$, then $D_{xt}\mathcal F(0,0)x_0$ is not in the range of $D_x\mathcal F(0,0)$.
			\end{itemize}
			Then there exists a neighborhood $U\subset X\times \mathbb R$ of $(0,0)$, $\varepsilon>0$ and functions $\sigma:(-\varepsilon,\varepsilon)\to X$ and $s:(-\varepsilon,\varepsilon)\to\mathbb R$ with $\sigma\not\equiv 0$ such that $\sigma(0)=0$, $s(0)=0$, and
			\begin{equation}
				\mathcal F^{-1}(0)\cap U=\left(\{(0,t):t\in\mathbb R\}\cup\{(\sigma(\alpha),s(\alpha)):|\alpha|<\varepsilon\}\right)\cap U.
			\end{equation}
			Moreover, if $\mathcal F_{xx}$ exists and is continuous, then $\sigma$ is continuously differentiable and $\sigma'(0)$ spans the null space of $D_x\mathcal F(0,0)$.
		\end{theorem}
		In the main text we proceed by checking  systematically  these four   properties for the operator \eqref{F}.

		\section{Appendix C}		
 Here we formulate  the  the Gearhart-Pr\"uss-Greiner (GPG) theorem \cite{EngNagBre2000}.
			
	\begin{theorem}\label{thm:GPG}
		 Let $X$ be a Hilbert space, and let $L:D(L)\to X$ be a linear operator, where the domain $D(L)$ of $L$ is a dense subspace of $X$. If the following conditions hold
		\begin{itemize}
			\item[(i)] the semigroup $(S(t))_{t\geq 0}$\label{sym3:S(t)} generated by $L$ is strongly continuous,
			\item[(ii)] The resolvent set of $L$ contains $\{z\in\mathbb C:\text{Re}\,z>0\}$, and
			\item[(iii)] The resolvent $(\lambda I-L)^{-1}$ is uniformly bounded on the above set, i.e.,
			\begin{equation}\label{eq:resolvent_bound}
				\sup_{\text{Re}\,\lambda>0}\Vert (\lambda I-L)^{-1}\Vert_X<\infty,
			\end{equation} 
		\end{itemize}
		then there exists $\Gamma,\sigma>0$ such that
		\begin{itemize}
			\item[(a)] For each $\lambda$ in the spectrum $\sigma(S(t))$, $|\lambda|<e^{-\sigma t}$, and 
			\item[(b)] For each $t\geq 0$, $\Vert S(t)\Vert_X\leq \Gamma e^{-\sigma t}$.
		\end{itemize}
	\end{theorem}
	In the main text we proceed by checking  systematically   these three   conditions for the operator \eqref{eq:SL_TW_linearization}.
	
	\bibliographystyle{plain}
	\bibliography{refs}
	
\end{document}